\documentclass[]{article}
\usepackage{graphicx}
\usepackage{amsfonts}
\usepackage{amsmath}
\usepackage{amssymb}
\usepackage{fancyhdr}
\usepackage{titlesec}
\usepackage{indentfirst}
\usepackage{booktabs}
\usepackage{verbatim}
\usepackage{color}
\usepackage{amsthm}
\usepackage{subcaption}

\usepackage[page,header]{appendix}
\usepackage{titletoc}

\newcommand{\rd}{\,\mathrm{d}}

\numberwithin{equation}{section}
\newtheorem{theorem}{Theorem}[section]

\newtheorem{corollary}[theorem]{Corollary}
\newtheorem{proposition}[theorem]{Proposition}

\begin{document}

\title{A particle method for the multispecies Landau equation\footnote{JAC was supported by the Advanced Grant Nonlocal-CPD (Nonlocal PDEs for Complex Particle Dynamics: Phase Transitions, Patterns and Synchronization) of the European Research Council Executive Agency (ERC) under the European Union's Horizon 2020 research and innovation programme (grant agreement No. 883363). JAC was also partially supported by the EPSRC grant numbers EP/T022132/1 and EP/V051121/1. JH and SVF's research was supported in part by AFOSR grant FA9550-21-1-0358 and DOE grant DE-SC0023164. SVF was also supported by the Pacific Institute for the Mathematical Sciences (PIMS). The research and findings may not reflect those of the Institute.}}

\author{Jos\'e A. Carrillo\footnote{Mathematical Institute, University of Oxford, Oxford OX2 6GG, UK (carrillo@maths.ox.ac.uk).}, \  \ Jingwei Hu\footnote{Department of Applied Mathematics, University of Washington, Seattle, WA 98195, USA (hujw@uw.edu).}, \  \  Samuel Q. Van Fleet\footnote{Department of Applied Mathematics, University of Washington, Seattle, WA 98195, USA (svfleet@uw.edu).}}
\maketitle

\begin{abstract}
    The multispecies Landau collision operator describes the two-particle, small scattering angle or grazing collisions in a plasma made up of different species of particles such as electrons and ions.  Recently, a structure preserving deterministic particle method \cite{single_species_Landau_particle_method} has been developed for the single species spatially homogeneous Landau equation.  This method relies on a regularization of the Landau collision operator so that an approximate solution, which is a linear combination of Dirac delta distributions, is well-defined.  Based on a weak form of the regularized Landau equation, the time dependent locations of the Dirac delta functions satisfy a system of ordinary differential equations.  In this work, we extend this particle method to the multispecies case, and examine its conservation of mass, momentum, and energy, and decay of entropy properties. We show that the equilibrium distribution of the regularized multispecies Landau equation is a Maxwellian distribution, and state a critical condition on the regularization parameters that guarantees a species independent equilibrium temperature. A convergence study comparing an exact multispecies BKW solution to the particle solution shows approximately 2nd order accuracy. Important physical properties such as conservation, decay of entropy, and equilibrium distribution of the particle method are demonstrated with several numerical examples.
\end{abstract}

{\small 


\section{Introduction}
\label{sec:Intro}
The multispecies Vlasov-Landau equation is a kinetic model that governs the evolution of collisional plasma made up of multiple different species of particles (commonly ions and electrons) and is given as 
\begin{equation}\label{eq:multispecies Vlasov-Landau equation}
\partial_{t}f_i + v\cdot \nabla_x f_i + \frac{q_i}{m_i}\left( E + v \times B\right) \cdot \nabla_v f_i = \sum_{j=1}^{s}Q_{ji}(f_j,f_i), \quad i = 1,...,s,
\end{equation}
where $f_i(t,x,v)$ is the number distribution function for particles of species $i$, that is, $f_i(t,x,v)$ gives the number of particles per unit volume with position $x \in \Omega \subset \mathbb{R}^d$ and velocity $v \in \mathbb{R}^d$ ($d=2$ or $3$).  $E$ and $B$ are the electric and magnetic fields either given externally or determined self-consistently via Maxwell's equations.  $q_i$ and $m_i$ are the charge and mass of particles of species $i$.  The Landau collision operator derived by Landau in \cite{L1936} is
\begin{equation}\label{eq:Landau operator}
        Q_{ji}\left(f_j,f_i\right)(v) = \nabla_v \cdot \int_{\mathbb{R}^{d}}A_{ji}\left(v-v_{*}\right)\left(\frac{1}{m_i}f_j(v_{*})\nabla_vf_i(v) - \frac{1}{m_j}f_i(v)\nabla_{v_{*}}f_j(v_{*})\right) \rd{v_{*}},
\end{equation}
with the kernel given by 
\begin{equation}\label{eq:kernel Landau operator}
    A_{ji}(z) = \frac{|\log \delta|}{8\pi \varepsilon_0^2}\frac{q_i^2q_j^2}{m_i}|z|^{\gamma} \left(|z|^2I_d - z \otimes z\right), \quad I_d  \text{ is the } d\times d  \text{ identity matrix},
\end{equation}
where $|\log{\delta}|$ is the Coulomb logarithm, $\epsilon_0$ is the vacuum permittivity and $-d-1\leq \gamma\leq 1$.  The case where $\gamma = 0$ is known as the Maxwell collision and $\gamma = -3$ is known as the Coulomb collision.  The Landau collision operator can be derived from the Boltzmann collision operator in the case of grazing collisions \cite{DLD1992,Desvillettes1992,Villani1998}.  


In this work, we focus exclusively on the spatially homogeneous version of (\ref{eq:multispecies Vlasov-Landau equation}), i.e.,
\begin{equation}\label{eq:homogenous Landau equation}
    \partial_tf_i = \sum_{j = 1}^s Q_{ji}\left(f_j,f_i \right), \quad f_i=f_i(t,v), \quad i = 1,...,s, 
\end{equation}
and develop a particle method for (\ref{eq:homogenous Landau equation}) following a similar idea in \cite{single_species_Landau_particle_method}, where a particle method is proposed for the single-species Landau equation. We will emphasize properties and features that particularly arise in the multispecies case. We note that recent work \cite{ZPH2022} also considers a similar extension of the particle method in \cite{single_species_Landau_particle_method} to the multispecies Landau equation. However, our work improves upon it in several aspects which we will highlight later in this section. 

To begin with, let us first review the basic properties of the multispecies Landau equation (\ref{eq:homogenous Landau equation}).  We rewrite \eqref{eq:Landau operator} in the ``$\log$" form
\begin{equation}\label{eq:Landau operator log form}
    Q_{ji}\left(f_j,f_i\right)(v) = \nabla_v \cdot \int_{\mathbb{R}^{d}}A_{ji}\left(v-v_{*}\right)\left(\frac{1}{m_i}\nabla_v \log{f_i} - \frac{1}{m_j}\nabla_{v_{*}}\log{f_{j*}}\right) f_{j*}f_i\rd{v_{*}},
\end{equation}
where $f_i := f_i(v)$ and $f_{j*} := f_{j}(v_*)$.  For the test function $\phi$, we can show that
\begin{equation}\label{eq:Landau weak form 1}
    \int_{\mathbb{R}^d}Q_{ji}(f_j,f_i) \phi\rd{v} = -\iint_{\mathbb{R}^{2d}} \nabla_v \phi  \cdot A_{ji}(v-v_{*})\left(\frac{1}{m_i}\nabla_v \log{f_i} - \frac{1}{m_j}\nabla_{v_{*}}\log{f_{j*}}\right) f_{j*}f_i\rd{v}\rd{v_{*}}.
\end{equation}
On the other hand, for the test function $\varphi$, we have
\begin{equation}\label{eq:Landau weak form 2}
    \int_{\mathbb{R}^d}Q_{ij}(f_i,f_j) \varphi\rd{v} = \iint_{\mathbb{R}^{2d}} \frac{m_i}{m_j}\nabla_{v_*} \varphi_*  \cdot A_{ji}(v-v_{*})\left(\frac{1}{m_i}\nabla_v \log{f_i} - \frac{1}{m_j}\nabla_{v_{*}}\log{f_{j*}}\right) f_{j*}f_i\rd{v}\rd{v_{*}},
\end{equation}
by switching the indices $i$ and $j$, and switching $v$ and $v_*$, and using $A_{ij}(z) = \frac{m_i}{m_j}A_{ji}(z)$.  Then, adding \eqref{eq:Landau weak form 1} and \eqref{eq:Landau weak form 2} we obtain the following weak form
\begin{equation}\label{eq:Landau operator averaged weak form}
\begin{aligned}
    &\int_{\mathbb{R}^d}Q_{ji}(f_j,f_i) \phi\rd{v} +  \int_{\mathbb{R}^d}Q_{ij}(f_i,f_j) \varphi \rd{v}  \\
   =& -\iint_{\mathbb{R}^{2d}} \left(\nabla_v \phi - \frac{m_i}{m_j}\nabla_{v_{*}}\varphi_*\right) \cdot A_{ji}(v-v_{*})\left(\frac{1}{m_i}\nabla_v \log{f_i} - \frac{1}{m_j}\nabla_{v_{*}}\log{f_{j*}}\right) f_{j*}f_i\rd{v}\rd{v_{*}}.
\end{aligned}
\end{equation} 
Using this weak form along with the fact that $A_{ji}(z)$ is positive semidefinite and $ A_{ji}(z)z = 0$, one can show that the solution to (\ref{eq:homogenous Landau equation}) satisfies the conservation of total mass, momentum, and energy:
\begin{align}
    &\frac{\rd}{\rd{t}} \sum_{i=1}^s \int_{\mathbb{R}^d}f_i\phi \rd{v} = 0, \quad \mbox{for} \ \phi = 1,m_iv,m_i|v|^2,\label{eq:total conservation of mass, momentum, and energy}
\end{align}
and the decay of total entropy:
\begin{align}
    &\frac{\rd}{\rd{t}} \sum_{i=1}^s \int_{\mathbb{R}^d}f_i\log{f_i} \rd{v} \leq 0,
\end{align}
with the equality obtained if and only if $f_i$ becomes the Maxwellian function:
\begin{equation}\label{eq:Landau equilibrium distribution}
    f_i = n_i\left(\frac{m_i}{2\pi T}\right)^{\frac{d}{2}}\exp{\left(-\frac{m_i|v-u|^2}{2T}\right)},
\end{equation}
where
\begin{equation}
    n_i = \int_{\mathbb{R}^d}f_i\rd{v}, \quad u = \frac{1}{\sum_{i=1}^{s}m_in_i}\sum_{i=1}^{s}m_i\int_{\mathbb{R}^d}f_iv\rd{v}, \quad T = \frac{1}{d\sum_{i=1}^sn_i}\sum_{i=1}^sm_i\int_{\mathbb{R}^d}f_i|v-u|^2\rd{v},
\end{equation}
are, respectively, the number density of species $i$, bulk velocity, and bulk temperature. 
 We refer to \cite{GZ2017} Theorem 4 for the proof.  

To apply the particle method to the homogeneous Landau equation \eqref{eq:homogenous Landau equation}, we first write it as a nonlinear transport equation
\begin{equation}\label{eq:Landau transport form}
    \partial_t f_i = \sum_{j = 1}^s Q_{ji}\left(f_j,f_i \right)=\nabla_v\cdot\left(\sum_{j=1}^{s}U_{ji}(f_j,f_i)f_i\right),
\end{equation}
where the velocity field is given by
\begin{equation}\label{eq:Landau velocity field introduction}
\sum_{j=1}^sU_{ji}\left(f_j,f_i\right) (v)=  \sum_{j=1}^{s}\int_{\mathbb{R}^{d}}A_{ji}\left(v-v_{*}\right)\left(\frac{1}{m_i}\nabla_v \log{f_i} - \frac{1}{m_j}\nabla_{v_{*}}\log{f_{j*}}\right) f_{j*}\rd{v_{*}}.
\end{equation}
The classical particle method \cite{Particle_method_review} seeks to approximate $f_i$ as a linear combination of Dirac delta distributions:
\begin{equation}\label{eq:particle method approximate solution}
f_i(t,v)\approx    f_i^N(t,v) = \sum_{p = 1}^N w^i_p\delta(v-v^i_p(t)),
\end{equation}
where $w_p^i$, $v_p^i$ are the particle weights and velocities of species $i$, and $N$ is the total number of particles used in species $i$.  

However, the $\log{f_i}$ and $\log{f_j}$ terms in \eqref{eq:Landau velocity field introduction} are not well-defined for Dirac delta functions and thus to proceed with the particle method, \eqref{eq:Landau velocity field introduction} must be regularized in some way. 
We follow the regularization strategy used for the nonlinear Fokker-Planck equations in \cite{CCP2019} and for the single-species Landau collision operator in \cite{single_species_Landau_particle_method}. This approach first recognizes that the $\nabla_v\log{f_i}$ term can be written as $\nabla_v\frac{\delta E_i}{\delta f_i}$, the gradient of the variational derivative of the Boltzmann entropy functional $E_i: = E(f_i)=\int_{\mathbb{R}^d}f_i\log{f_i}\rd{v}$, and then regularizes the entropy functional as
\begin{equation}\label{eq:regularized entropy}
    E^{\epsilon_i}_i:= E^{\epsilon_i}(f_i) = \int_{\mathbb{R}^d}  (f_i*\psi^{\epsilon_i})\log(f_i*\psi^{\epsilon_i})\rd{v},
\end{equation}
where $\psi^{\epsilon_i}$ is a  mollifier function that satisfies 
\begin{equation}\label{eq:mollifier function}
\psi^{\epsilon_i}(v) = \frac{1}{\epsilon_i^d}\psi\left(\frac{v}{\epsilon_i}\right), \quad \int_{\mathbb{R}^d}\psi(v)\rd{v} = 1, \quad \psi(v)=\psi(-v).
\end{equation} 
It is important to note that the regularization parameter $\epsilon_i > 0$ can be different for each species. 
For the regularized entropy, one can calculate that
\begin{equation}\label{eq:regularized variational derivative and gradient}
    \frac{\delta E_i^{\epsilon_i}}{\delta f_i} = \psi^{\epsilon_i}*\log{(\psi^{\epsilon_i}*f_i)}+1,  
    \quad \nabla_v\frac{\delta E_i^{\epsilon_i}}{\delta f_i} = (\nabla_v\psi^{\epsilon_i})*\log{(\psi^{\epsilon_i}*f_i)}.
\end{equation}
Correspondingly, the equation \eqref{eq:Landau transport form} is regularized as
\begin{equation}\label{eq:regularied Landau transport}
    \partial_tf_i = \sum_{j = 1}^s \Tilde{Q}_{ji}\left(f_j,f_i \right)=\nabla_v\cdot\left(\sum_{j=1}^s\Tilde{U}_{ji}(f_j,f_i)f_i\right),
\end{equation}
with
\begin{equation}\label{eq:regularized Landau operator}
    \Tilde{Q}_{ji}\left(f_j,f_i\right)(v) = \nabla_v \cdot \int_{\mathbb{R}^{d}}A_{ji}\left(v-v_{*}\right)\left(\frac{1}{m_i}\nabla_v \frac{\delta E_i^{\epsilon_i}}{\delta f_i} - \frac{1}{m_j}\nabla_{v_{*}}\frac{\delta E_{j*}^{\epsilon_j}}{\delta f_{j*}}\right) f_{j*}f_i\rd{v_{*}},
\end{equation}
and the velocity field
\begin{equation}\label{eq: Landau equation velocity field}
    \sum_{j=1}^s \Tilde{U}_{ji}\left(f_j,f_i\right)(v)= \sum_{j=1}^s \int_{\mathbb{R}^{d}}A_{ji}\left(v-v_{*}\right)\left(\frac{1}{m_i}\nabla_v \frac{\delta E_i^{\epsilon_i}}{\delta f_i} - \frac{1}{m_j}\nabla_{v_{*}}\frac{\delta E_{j*}^{\epsilon_{j}}}{\delta f_{j*}}\right) f_{j*}\rd{v_{*}}.
\end{equation}
This regularized equation can then invoke a particle solution \eqref{eq:particle method approximate solution}, whose particle velocities satisfy a large coupled ODE system:
\begin{equation}\label{eq:particle method ODE system} 
    \frac{\rd{v^i_p(t)}}{\rd{t}} = -\sum_{j=1}^{s}\Tilde{U}_{ji}(f^N_j,f^N_i)(v_p^i(t)).
\end{equation}

In the rest of this paper, we first study in Section \ref{sec:regularized landau} the structure of the regularized multi-species Landau equation (\ref{eq:regularied Landau transport}), characterizing its conservation properties, entropy decay structure, as well as the equilibrium distribution. Then in Section \ref{sec:Particle method}, we construct the full particle method for the regularized Landau equation, and show that the semi-discrete (continuous in time) method conserves total mass, momentum, and energy, as well as decays total entropy. We also discuss the conservation properties of the fully discrete method. Extensive numerical examples are presented in Section \ref{sec:numerical examples} to showcase the accuracy and structure-preserving properties of the method. 

Compared to the recent work \cite{ZPH2022}, the novelty of our work lies in the following: 1) We clearly identify the equilibrium of the regularized multispecies Landau equation as a Maxwellian function with species dependent temperature. This is in contrast to the true Maxwellian (\ref{eq:Landau equilibrium distribution}) of the original Landau equation where a unified temperature is reached for all species. This theoretical finding later becomes crucial in choosing the regularization parameters in the particle method so as to capture the correct long term behavior of the solution. This may partially answer the failure in a temperature relaxation case reported in \cite{ZPH2022}. 2) We implement both the forward Euler and implicit midpoint method. The latter exhibits perfect conservation of momentum and energy and is second order accurate, while in \cite{ZPH2022} a different first order method is used. 3) We construct an exact BKW solution to the multispecies Landau equation (see Appendix A) and use it to carefully study the order of accuracy of the particle method. This appears to be the first exact solution reported in the literature for the multispecies Landau equation and can be valuable for validating many numerical algorithms. In \cite{ZPH2022}, no convergence studies were made.

To conclude this section, we mention that besides the particle method, there are several Eulerian or mesh based methods for approximating the Landau equation, for example, finite difference, finite element, and spectral methods \cite{BUET1997310,HAGER2016644,Hirvijoki_2017,Lemou1998,DLD1994,Shiroto_2019,TAITANO2016391,TAITANO2015357,PARESCHI2000216}. Some of these methods can also be made structure-preserving. However, when dealing with multispecies case, it is commonly known that mesh based methods usually suffer from order deterioration for large mass ratios when a uniform mesh is used for all species, see, for instance \cite{JAISWAL201956}. As a remedy, some kind of adaptive mesh or coordinate transform is needed to maintain the accuracy for both heavy and light species \cite{TAITANO2016391}. We point out that the particle method proposed here has a clear advantage to handle the issue of large mass ratios as it is quite easy to sample different species with different resolution of particles, which we will demonstrate in our numerical examples.


\section{Regularized multispecies Landau equation}
\label{sec:regularized landau}


In this section, we study the structure of the regularized multi-species Landau
equation (\ref{eq:regularied Landau transport}), characterizing its conservation properties, entropy decay structure, and in particular identifying its equilibrium
distribution.

First of all, similarly as the derivation of (\ref{eq:Landau weak form 1}) and (\ref{eq:Landau operator averaged weak form}), one can easily obtain the following weak forms of the regularized Landau operator
\begin{equation} \label{weak1}
    \int_{\mathbb{R}^d}\tilde{Q}_{ji}(f_j,f_i) \phi\rd{v} = -\iint_{\mathbb{R}^{2d}} \nabla_v \phi  \cdot A_{ji}(v-v_{*})\left(\frac{1}{m_i}\nabla_v \frac{\delta E_i^{\epsilon_i}}{\delta f_i} - \frac{1}{m_j}\nabla_{v_{*}}\frac{\delta E_{j*}^{\epsilon_j}}{\delta f_{j*}}\right)f_{j*}f_i\rd{v}\rd{v_{*}}.
\end{equation}
\begin{equation} \label{weak2}
\begin{aligned}
    &\int_{\mathbb{R}^d}\tilde{Q}_{ji}(f_j,f_i) \phi\rd{v} +  \int_{\mathbb{R}^d}\tilde{Q}_{ij}(f_i,f_j) \varphi \rd{v}  \\
   =& -\iint_{\mathbb{R}^{2d}} \left(\nabla_v \phi - \frac{m_i}{m_j}\nabla_{v_{*}}\varphi_*\right) \cdot A_{ji}(v-v_{*})\left(\frac{1}{m_i}\nabla_v \frac{\delta E_i^{\epsilon_i}}{\delta f_i} - \frac{1}{m_j}\nabla_{v_{*}}\frac{\delta E_{j*}^{\epsilon_j}}{\delta f_{j*}}\right)f_{j*}f_i\rd{v}\rd{v_{*}}.
\end{aligned}
\end{equation} 
Using these weak forms, we can show the following.
\begin{proposition}\label{prop:regularized conservation and entropy decay}
    Assume the mollifier function $\psi^{\epsilon_i}(v)$ satisfies \eqref{eq:mollifier function}, 
     then the regularized Landau equation \eqref{eq:regularied Landau transport} satisfies
    \begin{enumerate}
        \item conservation of total mass, momentum, and energy:
        $$\frac{\rd}{\rd t}\sum_{i = 1}^s\int_{\mathbb{R}^d} f_i\phi\rd{v} = 0, \quad \phi = 1, m_iv, m_i|v|^2;$$
        \item decay of total regularized entropy 
        \begin{equation*}
        \begin{aligned}
            \frac{\rd }{\rd t}\sum_{i = 1}^{s} E^{\epsilon_i}_{i}=-\frac{1}{2}\sum_{i,j = 1}^s \iint_{\mathbb{R}^{2d}}
        & m_i\left(\frac{1}{m_i}\nabla_v\frac{\delta E^{\epsilon_i}_i}{\delta f_i}- \frac{1}{m_j} \nabla_{v_*} \frac{\delta E_{j*}^{\epsilon_j}}{\delta f_{j*}}\right)\\
        &\cdot A_{ji}(v-v_*)\left(\frac{1}{m_i}\nabla_v \frac{\delta E^{\epsilon_i}_i}{\delta f_i} - \frac{1}{m_j}\nabla_{v_{*}}\frac{\delta E_{j*}^{\epsilon_j}}{\delta f_{j*}}\right)f_{j*}f_i\rd{v}\rd{v_*} \leq 0,
        \end{aligned}
         \end{equation*}
         where $E^{\epsilon_i}_{i}$ is given by (\ref{eq:regularized entropy}).
    \end{enumerate}
\end{proposition}
\begin{proof}
\begin{enumerate}
Choosing proper test functions in (\ref{weak1})-(\ref{weak2}), we can show
\begin{align}
&\int_{\mathbb{R}^d}\tilde{Q}_{ji}(f_j,f_i)\rd{v}=0,\\
&\int_{\mathbb{R}^d}\tilde{Q}_{ji}(f_j,f_i)m_i v\rd{v}+\int_{\mathbb{R}^d}\tilde{Q}_{ij}(f_i,f_j)m_j v\rd{v}=0,\\
&\int_{\mathbb{R}^d}\tilde{Q}_{ji}(f_j,f_i)m_i |v|^2\rd{v}+\int_{\mathbb{R}^d}\tilde{Q}_{ij}(f_i,f_j)m_j |v|^2\rd{v}=0,
\end{align}
where we used $A_{ji}(z)z=0$ in the third equation. Then the conservation of mass, momentum and energy follows straightforwardly.

To see the entropy dissipation, we choose $\phi=\frac{\delta E^{\epsilon_i}_i}{\delta f_i}$ and $\varphi=\frac{\delta E^{\epsilon_j}_j}{\delta f_j}$ in (\ref{weak2}) to obtain
\begin{equation} \label{weak3}
\begin{aligned}
&\int_{\mathbb{R}^d}\tilde{Q}_{ji}(f_j,f_i) \frac{\delta E^{\epsilon_i}_i}{\delta f_i} \rd{v} +  \int_{\mathbb{R}^d}\tilde{Q}_{ij}(f_i,f_j) \frac{\delta E^{\epsilon_j}_j}{\delta f_j} \rd{v}  \\
   =& -\iint_{\mathbb{R}^{2d}} m_i\left( \frac{1}{m_i}\nabla_v\frac{\delta E^{\epsilon_i}_i}{\delta f_i}- \frac{1}{m_j} \nabla_{v_*} \frac{\delta E_{j*}^{\epsilon_j}}{\delta f_{j*}}\right) \cdot A_{ji}(v-v_{*})\left(\frac{1}{m_i}\nabla_v \frac{\delta E_i^{\epsilon_i}}{\delta f_i} - \frac{1}{m_j}\nabla_{v_{*}}\frac{\delta E_{j*}^{\epsilon_j}}{\delta f_{j*}}\right)f_{j*}f_i\rd{v}\rd{v_{*}}\leq 0,
\end{aligned}
\end{equation} 
since $A_{ji}$ is positive semidefinite.
Then
    \begin{equation*}
    \begin{aligned}
    &\frac{\rd{}}{\rd{t}} \sum_{i = 1}^{s}E_i^{\epsilon_i} = \frac{\rd{}}{\rd{t}} \sum_{i=1}^s \int_{\mathbb{R}^d}(f_i*\psi^{\epsilon_i})\log{(f_i*\psi^{\epsilon_i})}\rd{v} = \sum_{i=1}^s\int_{\mathbb{R}^d} (\partial_tf_i*\psi^{\epsilon_i})\left(\log{(f_i*\psi^{\epsilon_i}})+1\right)\rd{v}\\    =&\sum_{i=1}^s\int_{\mathbb{R}^d}\partial_tf_i\left(\psi^{\epsilon_i}*(\log{(\psi^{\epsilon_i}*f_i)}+1 )\right)\rd{v}=\sum_{i=1}^s\int_{\mathbb{R}^d}\partial_tf_i\left(\psi^{\epsilon_i}*\log{(\psi^{\epsilon_i}*f_i)}+1 \right)\rd{v} \\
    =&\sum_{i=1}^s \int_{\mathbb{R}^d}\partial_t f_i\frac{\delta E^{\epsilon_i}_i}{\delta f_i}\rd{v} =\sum_{i,j=1}^s \int_{\mathbb{R}^d} \tilde{Q}_{ji}(f_j,f_i)\frac{\delta E^{\epsilon_i}_i}{\delta f_i}\rd{v}\\
    =& \frac{1}{2} \sum_{i,j=1}^n \left(\int_{\mathbb{R}^d}\tilde{Q}_{ji}(f_j,f_i) \frac{\delta E^{\epsilon_i}_i}{\delta f_i} \rd{v} +  \int_{\mathbb{R}^d}\tilde{Q}_{ij}(f_i,f_j) \frac{\delta E^{\epsilon_j}_j}{\delta f_j} \rd{v} \right),
    \end{aligned}
\end{equation*}
where we used properties of the mollifier function in passing the equalities in line 2. Finally applying (\ref{weak3}) yields the desired inequality.
\end{enumerate}
\end{proof}

Many mollifier functions can satisfy the condition in \eqref{eq:mollifier function}. Two examples are generalized Gaussian functions and compactly supported functions such as B-splines as mentioned in \cite{Particle_method_review}. In the following, we try to quantify the equilibrium distribution of the regularized Landau equation (\ref{eq:regularied Landau transport}). This would rely crucially on the fact that the mollifier is a Gaussian kernel
\begin{equation}\label{eq:gaussian mollifier}
    \psi^{\epsilon_i}(v) = \frac{1}{(2\pi\epsilon_i)^{\frac{d}{2}}}\exp\left(-\frac{|v|^2}{2\epsilon_i} \right), \quad \epsilon_i>0,
\end{equation}
which will be assumed for the rest of this section.


We first show that the variational derivative of the regularized entropy is a quadratic polynomial at the equilibrium. The technique we use follows Theorem 4 in \cite{GZ2017}.

\begin{proposition}\label{prop:steady state quadratic} If $f_i$ is an equilibrium distribution to (\ref{eq:regularied Landau transport}) or equivalently, $\frac{1}{m_i}\nabla_{v}\frac{\delta E_i^{\epsilon_i}}{\delta f_i} - \frac{1}{m_j}\nabla_{v_{*}}\frac{\delta E_{j*}^{\epsilon_j}}{\delta f_{j*}}$ is in the kernel of $A_{ji}(v-v_*)$ then
\begin{equation}
    \frac{\delta E_i^{\epsilon_i}}{\delta f_i} = \lambda_i^{(0)} + m_i\lambda^{(1)} \cdot v + m_i\lambda^{(2)}\frac{|v|^2}{2},
\end{equation}
where the constants $\lambda^{(0)}_i,\lambda^{(2)} \in \mathbb{R}$, and $\lambda^{(1)} \in \mathbb{R}^d$ are determined from the conserved macroscopic quantities.
\end{proposition}
\begin{proof}
     By the definition of $A_{ji}(v-v_*)$, a vector belongs to its kernel if it is linearly dependent with $v-v_*$, thus there exists $\lambda_{ij}^{(2)} \in \mathbb{R}$ such that 
     \begin{equation}\label{eq:steady state linear dependence 1}
         \frac{1}{m_i}\nabla_{v}\frac{\delta E_i^{\epsilon_i}}{\delta f_i} - \frac{1}{m_j}\nabla_{v_{*}}\frac{\delta E_{j*}^{\epsilon_j}}{\delta f_{j*}} = \lambda_{ij}^{(2)}(v,v_*)(v-v_*).
     \end{equation}
     Switching the indices $i$ and $j$ along with $v$ and $v_*$ we have
     \begin{equation}\label{eq:steady state linear dependence 2}
         \frac{1}{m_i}\nabla_{v}\frac{\delta E_i^{\epsilon_i}}{\delta f_i} - \frac{1}{m_j}\nabla_{v_{*}}\frac{\delta E_{j*}^{\epsilon_j}}{\delta f_{j*}} = \lambda_{ji}^{(2)}(v_*,v)(v-v_*).
     \end{equation}
     Subtracting \eqref{eq:steady state linear dependence 1} from \eqref{eq:steady state linear dependence 2} results in 
     \begin{equation}
         \lambda_{ij}^{(2)}(v,v_*) = \lambda_{ji}^{(2)}(v_*,v),
     \end{equation}
     and for $v = v_*$, 
\begin{equation}\label{eq:lambda ij symmetry}
         \lambda^{(2)}_{ij}(v,v) = \lambda^{(2)}_{ji}(v,v).
     \end{equation}
     Let $k,\ell \in \{1,...,d\}$ and examine the $k$th element of the derivative of \eqref{eq:steady state linear dependence 1} with respect to $v_{\ell}$
     \begin{equation}
         \frac{1}{m_i}\partial_{ v_{\ell}}\partial_{v_k}\frac{\delta E^{\epsilon_i}_i}{\delta f_i} = \partial_{v_{\ell}} \lambda_{ij}^{(2)}(v,v_*)(v_k-v_{k*}) + \lambda^{(2)}_{ij}(v,v_*)\delta_{k \ell},
     \end{equation}
     and for $v = v_*$,
     \begin{equation}\label{eq:lambda j dependence}
         \frac{1}{m_i}\partial_{ v_{\ell}}\partial_{v_k}\frac{\delta E_i^{\epsilon_i}}{\delta f_i} =  \lambda^{(2)}_{ij}(v,v)\delta_{k \ell}.
     \end{equation}
     Let $n \in \{1,...,d\}$ and differentiate the equation above with respect to $v_n$, then switch the indices $n$ and $k$ to see the following two equations
     \begin{equation}
         \frac{1}{m_i}\partial_{v_n}\partial_{v_\ell}\partial_{v_k}\frac{\delta E_i^{\epsilon_i}}{\delta f_i} = \partial_{v_n}\lambda_{ij}^{(2)}(v,v)\delta_{k\ell}, \quad \mbox{and} \quad
         \frac{1}{m_i}\partial_{v_k}\partial_{v_\ell}\partial_{v_n}\frac{\delta E_i^{\epsilon_i}}{\delta f_i} = \partial_{v_k}\lambda_{ij}^{(2)}(v,v)\delta_{n\ell},
     \end{equation}
     which leads to 
    \begin{equation}
        \partial_{v_k}\lambda_{ij}^{(2)}(v,v)\delta_{n\ell} = \partial_{v_n}\lambda_{ij}^{(2)}(v,v)\delta_{k\ell}.
    \end{equation}
    Consider the case where $n = \ell \neq k$ to see that 
    \begin{equation}
        \partial_{v_k}\lambda_{ij}^{(2)}(v,v) = 0,
    \end{equation}
    which implies that $\lambda_{ij}^{(2)}(v,v)$ is a constant.  Additionally, equation \eqref{eq:lambda ij symmetry} tells us that $\lambda_{ij}^{(2)}$ is symmetric in $i$ and $j$, while equation \eqref{eq:lambda j dependence} tells us $\lambda_{ij}^{(2)}$ does not depend on $j$.  These three properties of $\lambda^{(2)}_{ij}(v,v)$, allow for \eqref{eq:lambda j dependence} to be written as
    \begin{equation}
    \partial_{ v_{\ell}}\partial_{v_k}\frac{\delta E_i^{\epsilon_i}}{\delta f_i} =  m_i\lambda^{(2)}\delta_{k \ell}.
    \end{equation}
    Integrating twice yields 
    \begin{equation}\label{eq:quadratic entropy 1}
        \frac{\delta E_i^{\epsilon_i}}{\delta f_i} =  \lambda_i^{(0)} + \lambda_i^{(1)}\cdot v + m_i \lambda^{(2)} \frac{|v|^2}{2} ,
    \end{equation}
    and substituting \eqref{eq:quadratic entropy 1} into \eqref{eq:steady state linear dependence 1} one can show that $\frac{1}{m_i}\lambda_i^{(1)} = \frac{1}{m_j}\lambda_j^{(1)}$, meaning we can define $\lambda^{(1)} = \frac{1}{m_i}\lambda^{(1)}_i$ and rewrite \eqref{eq:quadratic entropy 1} as
\begin{equation}
        \frac{\delta E^{\epsilon_i}_i}{\delta f_i} = \lambda_i^{(0)} + m_i\lambda^{(1)} \cdot v + m_i\lambda^{(2)}\frac{|v|^2}{2}.
    \end{equation}
\end{proof}
Next we follow a similar technique in \cite{single_species_Landau_particle_method} (lemma 7) to show that if $ \frac{\delta E^{\epsilon_i}_i}{\delta f_i} = \lambda_i^{(0)} +m_i\lambda^{(1)} \cdot v + m_i\lambda^{(2)}\frac{|v|^2}{2}$, then $f_i$ must be a Maxwellian.  

\begin{proposition}\label{Prop:Maxwellian steady state}
    If $f_i$ satisfies the following equation, 
    \begin{equation}\label{eq:quadratic entropy 2}
       \frac{\delta E^{\epsilon_i}_i}{\delta f_i} = \lambda_i^{(0)} +m_i\lambda^{(1)} \cdot v + m_i\lambda^{(2)}\frac{|v|^2}{2} ,
    \end{equation}
    then the equilibrium distribution to the regularized multispecies Landau equation (\ref{eq:regularied Landau transport}) is a Maxwellian function of the form
    \begin{equation} \label{eq:equilibrium regularized Landau}
        f_i(v) = n_i\left(\frac{m_i}{2\pi T_i}\right)^{\frac{d}{2}}\exp{\left(-\frac{m_i|v-u|^2}{2T_i} \right)},
    \end{equation}
    where the density, velocity, and temperature are given by
\begin{equation}\label{eq:Equilibrium distribution particle number, relaxation velocity, and relaxation temperature}
    \begin{aligned}
        n_i &= \left(-\frac{2\pi}{m_i\lambda^{(2)}}\right)^{\frac{d}{2}}\exp{\left(\lambda_i^{(0)}-\frac{dm_i\epsilon_i\lambda^{(2)}}{2}- \frac{m_i^2\epsilon_i\left|\lambda^{(1)}\right|^2}{2\left(1+m_i\epsilon_i\lambda^{(2)}\right)}-\frac{m_i\left|\lambda^{(1)}\right|^2}{2\lambda^{(2)}\left(1+m_i\epsilon_i\lambda^{(2)}\right)}\right)},\\
        u &= -\frac{\lambda^{(1)}}{\lambda^{(2)}},\\
        T_i & = -\left(m_i\epsilon_i + \frac{1}{\lambda^{(2)}}\right).
    \end{aligned}
    \end{equation}
\end{proposition}

\begin{proof}
    Using \eqref{eq:regularized variational derivative and gradient} and \eqref{eq:quadratic entropy 2} we begin with 
    \begin{equation}        \psi^{\epsilon_i}*\log{\left(f_i*\psi^{\epsilon_i}\right)}(v) =        
        \lambda_i^{(0)} +m_i\lambda^{(1)} \cdot v + m_i\lambda^{(2)}\frac{|v|^2}{2} ,
    \end{equation}
    where $\lambda_i^{(0)}$ has been recycled to include the $+1$ term from \eqref{eq:regularized variational derivative and gradient}.  Taking the Fourier transform of both sides (the Fourier transform of $f$ is denoted as $\hat{f}$) yields
    \begin{equation}
        \left(\log{\left(f_i*\psi^{\epsilon_i} \right)}\right)^{\widehat{}}(y)  = \frac{1}{\hat{\psi}^{\epsilon_i}(y)}\left(\lambda_i^{(0)} \delta(y) + Im_i\lambda^{(1)}\cdot \nabla \delta(y) - m_i\frac{\lambda^{(2)}}{2}\Delta \delta(y)\right),
    \end{equation}
    where $I = \sqrt{-1}$ since $i$ is already used as an index and $\delta(y)$ is the Dirac delta distribution.  Taking the inverse Fourier transform of both sides we have
    \begin{equation}
        \log{\left(f_i*\psi^{\epsilon_i}(v)\right)} = \lambda_i^{(0)} - \frac{dm_i\epsilon_i\lambda^{(2)}}{2}  + m_i\lambda^{(1)}\cdot v+ m_i\frac{\lambda^{(2)}}{2} |v|^2,
    \end{equation}
    and taking the exponential of both sides gives
    \begin{equation}
        f_i*\psi^{\epsilon_i}(v) = \exp{\left(\lambda_i^{(0)}-\frac{dm_i \epsilon_i \lambda^{(2)}}{2}-\frac{m_i\left|\lambda^{(1)}\right|^2}{2\lambda^{(2)}}\right)}\exp{\left(\frac{m_i\lambda^{(2)}}{2}\left|v + \frac{\lambda^{(1)}}{\lambda^{(2)}}\right|^2 \right)}.
    \end{equation}
    Again taking the Fourier transform of both sides, a lengthy calculation yields 
    \begin{equation}
    \begin{aligned}
        \hat{f_i}(y) = \left(\frac{-1}{m_i\lambda^{(2)}} \right)^{\frac{d}{2}}
        &\exp{\left(\lambda_i^{(0)}-\frac{dm_i\epsilon_i\lambda^{(2)}}{2}-\frac{m_i^2\epsilon_i\left|\lambda^{(1)}\right|^2}{2\left(1+m_i\epsilon_i\lambda^{(2)}\right)}\right)}\\
        &\times\exp{\left(\frac{1+m_i\epsilon_i\lambda^{(2)}}{2m_i\lambda^{(2)}}\left|y+I\frac{m_i\lambda^{(1)}}{1+m_i\epsilon_i\lambda^{(2)}}\right|^2\right)}.
    \end{aligned}
    \end{equation}
    Finally, one last inverse Fourier transform shows that 
    \begin{equation}
    \begin{aligned}
        f_i(v) = \left(\frac{1}{1+m_i\epsilon_i\lambda^{(2)}} \right)^{\frac{d}{2}}&
        \exp{\left(\lambda_i^{(0)}-\frac{dm_i\epsilon_i\lambda^{(2)}}{2}- \frac{m_i^2\epsilon_i\left|\lambda^{(1)}\right|^2}{2\left(1+m_i\epsilon_i\lambda^{(2)}\right)}-\frac{m_i\left|\lambda^{(1)}\right|^2}{2\lambda^{(2)}\left(1+m_i\epsilon_i\lambda^{(2)}\right)}\right)}\\
        &\times
        \exp{\left(m_i\left(\frac{\lambda^{(2)}}{2(m_i\epsilon_i\lambda^{(2)}+1)}\right)\left| v+\frac{\lambda^{(1)}}{\lambda^{(2)}}\right|^2\right)},
    \end{aligned}
    \end{equation}
    which is in the form of (\ref{eq:equilibrium regularized Landau}) with (\ref{eq:Equilibrium distribution particle number, relaxation velocity, and relaxation temperature}).
\end{proof}
Upon comparison of the true equilibrium \eqref{eq:Landau equilibrium distribution} of the original Landau equation with \eqref{eq:equilibrium regularized Landau} of the regularized Landau equation, we find that the temperature in \eqref{eq:Landau equilibrium distribution} does not depend on species $i$ while the temperature in \eqref{eq:equilibrium regularized Landau} is generally species dependent.  This leads to a natural condition on the regularization parameter $\epsilon_i$ in order to guarantee a species independent equilibrium temperature.
\begin{corollary}\label{corollary: Universal temperature relaxation}
    If $m_i\epsilon_i = \epsilon$ for $i = 1,...,s$, the temperature of the Maxwellian function given in \eqref{eq:equilibrium regularized Landau} can be made species independent and is given by 
    \begin{equation*}
        T_i = T = -\left(\epsilon+ \frac{1}{\lambda^{(2)}}\right), \quad \mbox{for} \quad i = 1,...,s.
    \end{equation*}
\end{corollary}

\section{Particle method for the multispecies Landau equation}
\label{sec:Particle method}

In this section, we present the particle method for the regularized multispecies Landau equation and show its conservation and entropy decay properties.  Recall that the Landau equation is written as a nonlinear transport equation \eqref{eq:regularied Landau transport} with the velocity field given by \eqref{eq: Landau equation velocity field}, and the particle solution is a linear combination of Dirac delta functions \eqref{eq:particle method approximate solution}, where the weights $w_p^i$ are fixed and the particle velocities satisfy the following system obtained by substituting \eqref{eq:particle method approximate solution} into \eqref{eq:particle method ODE system}:
\begin{equation}\label{eq:semi-discrete particle method} 
\begin{aligned}
    \frac{\rd{v^i_p(t)}}{\rd{t}} & = -\sum_{j=1}^s\Tilde{U}_{ji}(f^N_j,f^N_i)(v_p^i(t)) \\
    & = -\sum_{j=1}^s\sum_{q=1}^{N}w_q^jA_{ji}(v_p^i - v_q^j)\left(\frac{1}{m_i}\nabla \frac{\delta E_i^{\epsilon_i,N}}{\delta f_i}(v_p^i) - \frac{1}{m_j}\nabla\frac{\delta E_j^{\epsilon_j,N}}{\delta f_j}(v_q^j)\right).
\end{aligned}
\end{equation}
The regularized entropy functional, its variational derivative, and gradient of the variational derivative at the particle level are obtained by substituting \eqref{eq:particle method approximate solution} into \eqref{eq:regularized entropy} and \eqref{eq:regularized variational derivative and gradient}
\begin{align}
        E^{\epsilon_i,N}_{i} :&= \int_{\mathbb{R}^d}\sum_{p=1}^Nw_p^i\psi^{\epsilon_i}(v-v_p^i)\log{\left(\sum_{r=1}^Nw_r^i \psi^{\epsilon_i}(v - v_r^i) \right)}\rd{v}, \label{eq:semidiscrete entropy}\\
        \frac{\delta E_i^{{\epsilon_i,N}}}{\delta f_i^N}(v_p^i) &:=\int_{\mathbb{R}^d}\psi^{\epsilon_i}(v_p^i-v)\log{ \left(\sum_{r=1}^Nw_r^i\psi^{\epsilon_i}(v - v_r^i) \right)}\rd{v} + 1,\label{eq:semidiscrete variational derivative of entropy}\\
        \nabla\frac{\delta E_i^{{\epsilon_i,N}}}{\delta f_i^N}(v_p^i) &:=\int_{\mathbb{R}^d}\nabla\psi^{\epsilon_i}(v_p^i-v)\log{ \left(\sum_{r=1}^Nw_r^i\psi^{\epsilon_i}(v - v_r^i) \right)}\rd{v}\label{eq:semidiscrete gradient of the variational derivative of entropy}.
\end{align}
The macroscopic quantities: species number density, mass density, velocity and temperature are defined at the particle level as
\begin{equation}\label{eq:macroscopic quantities}
    n_i = \sum_{p = 1}^Nw_p^i,\quad \rho_i = m_in_i, \quad u_i = \frac{1}{n_i}\sum_{p = 1}^N w_p^iv_p^i, \quad T_i = \frac{m_i}{dn_i}\sum_{p=1}^N w_p^i|v_p^i-u_i|^2,
\end{equation} 
and the total number density, mass density, velocity, and temperature are 
\begin{equation}\label{eq:total macroscopic quantities }
    n = \sum_{i = 1}^s n_i, \quad \rho = \sum_{i=1}^s\rho_i, \quad u = \frac{1}{\rho}\sum_{i=1}^s\rho_iu_i, \quad  T = \frac{1}{dn}\sum_{i=1}^sm_i\sum_{p=1}^Nw_p^i|v_p^i - u|^2.
\end{equation}
\begin{proposition}\label{prop:semi-discrete particle method conservation}
At the semi-discrete (continuous in time) level, the particle method \eqref{eq:semi-discrete particle method} 
\begin{enumerate}
    \item conserves total mass, momentum, and  energy:
    $$\frac{\rd{}}{\rd{t}} \sum_{i=1}^s \sum_{p = 1}^N w_p^i\phi_i(v_p^i) = 0, \quad \phi_i(v) = 1, m_iv, m_i|v|^2.$$
    \item decays total regularized entropy:
    \begin{equation*}
    \begin{aligned}
        \frac{\rd}{\rd t}\sum_{i=1}^s E_i^{\epsilon_i, N} 
                        = -\frac{1}{2}\sum_{i,j=1}^s\sum_{p,q = 1}^N &w_p^iw^j_q \left(\nabla\frac{\delta E^{{\epsilon_i,N}}_i}{\delta f^N_i}(v_p^i)-\frac{m_i}{m_j}\nabla\frac{\delta E^{\epsilon_j,N}_j}{\delta f^N_j}(v_q^j)\right)\\
        & \cdot A_{ji}(v^i_p-v^j_q)\left(\frac{1}{m_i}\nabla\frac{\delta E^{\epsilon_i,N}_i}{\delta f_i}(v_p^i) - \frac{1}{m_j}\nabla\frac{\delta E^{\epsilon_j,N}_j}{\delta f_j}(v_q^j) \right) \leq 0.
        \end{aligned}
        \end{equation*}
\end{enumerate}
\begin{proof}
    \begin{enumerate}
        \item    \begin{equation*}
            \begin{aligned}
               & \frac{\rd}{\rd t} \sum_{i=1}^s\sum_{p = 1}^N w_p^i\phi_i(v_p^i) = \sum_{i=1}^s\sum_{p = 1}^N w_p^i \nabla \phi_i(v_p^i) \cdot \frac{\rd{v_p^i}}{\rd{t}} \\
                =& -\sum_{i,j=1}^s\sum_{p,q = 1}^N\  w_p^iw^j_q \nabla \phi_i(v_p^i) \cdot A_{ji}(v^i_p-v^j_q)\left(\frac{1}{m_i}\nabla\frac{\delta E^{\epsilon_i,N}_i}{\delta f_i}(v_p^i) - \frac{1}{m_j}\nabla\frac{\delta E^{\epsilon_j,N}_j}{\delta f_j}(v_q^j) \right)\\
                = &\sum_{i,j=1}^s\sum_{p,q = 1}^N\ \frac{m_i}{m_j}w_p^iw^j_q \nabla \phi_j(v_q^j)
                \cdot A_{ji}(v^i_p-v^j_q)\left(\frac{1}{m_i}\nabla\frac{\delta E^{\epsilon_i,N}_i}{\delta f_i}(v_p^i) - \frac{1}{m_j}\nabla\frac{\delta E^{\epsilon_j,N}_j}{\delta f_j}(v_q^j) \right)\\
                    =& -\frac{1}{2}\sum_{i,j=1}^s\sum_{p,q = 1}^N w_p^iw^j_q \left(\nabla \phi_i(v_p^i)-\frac{m_i}{m_j}\nabla \phi_j(v_q^j) \right)\cdot A_{ji}(v^i_p-v^j_q)\left(\frac{1}{m_i}\nabla\frac{\delta E^{\epsilon_i,N}_i}{\delta f_i}(v_p^i) - \frac{1}{m_j}\nabla\frac{\delta E^{\epsilon_j,N}_j}{\delta f_j}(v_q^j) \right).
                \end{aligned}
            \end{equation*}
        The second to last equality comes from switching the indices $i$ and $j$, and $p$ and $q$, and using $A_{ij}=\frac{m_i}{m_j}A_{ji}$. The last equality comes from the familiar process of averaging the third and fourth line.  Total conservation of momentum and energy ($\phi_i(v) = m_iv$ and $\phi_i(v) = m_iv^2$) are achieved because $A_{ji}(z)z = 0$.  For total conservation of mass $(\phi_i(v) = 1)$, only the second line is necessary.
        \item 
        \begin{equation*}
            \begin{aligned}
            \frac{\rd{}}{\rd{t}} \sum_{i = 1}^s E^{\epsilon_i,N}_{i} &= \frac{\rd}{\rd{t}}\sum_{i=1}^s\int_{\mathbb{R}^d}\sum_{p=1}^Nw_p^i\psi^{\epsilon_i}(v-v_p^i)\log{\left(\sum_{r=1}^Nw_r^i \psi^{\epsilon_i}(v - v_r^i) \right)}\rd{v}\\
            & = \sum_{i=1}^{s}\int_{\mathbb{R}^d}\sum_{p=1}^Nw_p^i\nabla \psi^{\epsilon_i}(v_p^i-v)\cdot\frac{\rd{v_p^i}}{\rd{t}}\log{\left(\sum_{r=1}^Nw_r^i \psi^{\epsilon_i}(v - v_r^i) \right)}\rd{v}\\
            & \quad + \sum_{i=1}^s\int_{\mathbb{R}^d}\sum_{r=1}^Nw_r^i\nabla \psi^{\epsilon_i}(v_r^i-v)\cdot\frac{\rd{v_r^i}}{\rd{t}} =: I_1 + I_2.
            \end{aligned} 
        \end{equation*}
        First, we consider $I_2$
        \begin{equation*}
            I_2 = \sum_{i=1}^s\int_{\mathbb{R}^d}\sum_{r=1}^Nw_r^i\nabla \psi^{\epsilon_i}(v_r^i-v)\cdot\frac{\rd{v_r^i}}{\rd{t}} = \frac{\rd}{\rd{t}}\sum_{i=1}^s\sum_{r=1}^Nw_r^i \int_{\mathbb{R}^d}\psi^{\epsilon_i}(v-v_r^i)\rd{v} = 0,
        \end{equation*}
        since $\int_{\mathbb{R}^d}\psi^{\epsilon_i}(v-v_r^i)\rd{v} = 1$.  Using \eqref{eq:semidiscrete gradient of the variational derivative of entropy}, $I_1$ can be written as
        \begin{align*}
            I_1 &= \sum_{i=1}^{s}\sum_{p=1}^Nw_p^i\left(\int_{\mathbb{R}^d}\nabla \psi^{\epsilon_i}(v_p^i-v)\log{\left(\sum_{r=1}^Nw_r^i \psi^{\epsilon_i}(v - v_r^i) \right)}\rd{v}\right) \cdot \frac{\rd{v_p^i}}{\rd{t}} \\
            &= \sum_{i=1}^s\sum_{p=1}^Nw_p^i\left(\nabla\frac{\delta E^{{\epsilon_i,N}}_i}{\delta f^N_i}(v_p^i)\right) \cdot \frac{\rd{v_p^i}}{\rd{t}}\\
            & = -\sum_{i,j=1}^s\sum_{p,q=1}^Nw_p^iw_q^j\left(\nabla\frac{\delta E^{{\epsilon_i,N}}_i}{\delta f^N_i}(v_p^i)\right) \cdot A_{ji}(v_p^i - v_q^j)\left(\frac{1}{m_i}\nabla \frac{\delta E_i^{{\epsilon_i,N}}}{\delta f_i}(v_p^i) - \frac{1}{m_j}\nabla\frac{\delta E_j^{\epsilon_j,N}}{\delta f_j}(v_q^j)\right)\\
            &
            \begin{aligned}
                = -\frac{1}{2}\sum_{i,j=1}^s\sum_{p,q=1}^N
                &w_p^iw_q^j\left(\nabla\frac{\delta E^{{\epsilon_i,N}}_i}{\delta f^N_i}(v_p^i)-\frac{m_i}{m_j}\nabla\frac{\delta E^{\epsilon_j,N}_j}{\delta f^N_j}(v_q^j)\right)\\
                &\cdot A_{ji}(v_p^i - v_q^j) \left(\frac{1}{m_i}\nabla \frac{\delta E_i^{\epsilon_i,N}}{\delta f_i}(v_p^i) - \frac{1}{m_j}\nabla\frac{\delta E_j^{\epsilon_j,N}}{\delta f_j}(v_q^j)\right)\leq 0,
                \end{aligned}
        \end{align*}
        since $A_{ji}$ is positive semidefinite. 
    \end{enumerate}
\end{proof}
\end{proposition}
\subsection{Initialization and mesh}\label{subsec:Initialization and mesh}

To initialize the particle method, we consider square computational domains and without loss of generality we assume they are centered at the origin. Specifically, for each species $i$, we choose the domain as $[-L_i,L_i]^d$, $L_i>0$. The interval $[-L_i,L_i]$ is divided into $n$ equally spaced subintervals with length $h_i = 2L_i/n$. Using these subdivisions, the $[-L_i,L_i]^d$ is divided into $n^d = N$ elements with uniform size $h_i^d$. 

For the initial condition $f_i(0,v) = f_i^{0}(v)$, we approximate it as
\begin{equation}\label{eq:initial condition}
    f_i^N(0,v) = \sum_{p=1}^Nw_p^i\delta(v-v_p^i(0)),\quad v_p^i(0) = v_{h_i},\quad w_p^i = h_i^df_i^0(v_{h_i}),
\end{equation}
where $v_{h_i}$ denotes the center of each element, and a midpoint quadrature is used to approximate the weight in each element.

Furthermore, the midpoint rule is also used to approximate the integrals in (\ref{eq:semidiscrete entropy}) and in (\ref{eq:semidiscrete gradient of the variational derivative of entropy}), that is,
\begin{equation}
\nabla\frac{\delta E_i^{{\epsilon_i,N}}}{\delta f_i^N}(v_p^i)\approx h_i^d \sum_{h_i}\nabla\psi^{\epsilon_i}(v_p^i-v_{h_i})\log{ \left(\sum_{r=1}^Nw_r^i\psi^{\epsilon_i}(v_{h_i} - v_r^i) \right)}:=F_i^{\epsilon_i,N}(v_p^i),
\end{equation}
\begin{equation}\label{eq:discrete-in-velocity entropy}
    E^{{\epsilon_i,N}}_i  \approx h_i^d\sum_{h_i}\sum_{p=1}^Nw^i_p\psi^{\epsilon_i}(v_{h_i} - v_p^i)\log{\left(\sum_{r=1}^{N}w^i_r\psi^{\epsilon_i}(v_{h_i}-v_r^i)\right)}.
\end{equation}
The resulting particle method then reads
\begin{equation}\label{eq:discrete-in-velocity particle method} 
    \frac{\rd{v^i_p(t)}}{\rd{t}}  = -\sum_{j=1}^s\sum_{q=1}^{N}w_q^jA_{ji}(v_p^i - v_q^j)\left(\frac{1}{m_i}F_i^{\epsilon_i,N}(v_p^i) - \frac{1}{m_j}F_j^{\epsilon_j,N}(v_q^j)\right).
\end{equation}
One can show that this method still conserves mass, momentum, and energy. The entropy decays almost in time with $O(h^2)$ error. These properties can be shown with a similar technique used for Proposition \ref{prop:semi-discrete particle method conservation} along with the fact that the midpoint rule is a second order accurate method. We omit the detail.

Finally, given $\{v_p^i(t)\}$, in order to reconstruct a regularized solution, we can convolve
the particle solution \eqref{eq:particle method approximate solution} with the mollifier \eqref{eq:mollifier function},
\begin{equation}\label{eq:blob solution}
    \tilde{f}^{N}_i(t,v) := (\psi^{\epsilon_i}*f^N_i)(t,v) = \sum_{p=1}^Nw_p^i\psi^{\epsilon_i}(v-v_p^i(t)).
\end{equation}

\subsection{Time discretization}
\label{subsec:time}
The particle velocities at time $t$ are obtained by solving the system of ODEs \eqref{eq:discrete-in-velocity particle method}.  In \cite{single_species_Landau_particle_method}, it is shown that the forward Euler method conserves mass and momentum exactly, while energy is conserved up to $O(\Delta t)$.  The same can be said about the forward Euler method for the multi-species case regarding total mass, momentum and energy while also providing a simple, relatively low cost time update calculation.  In \cite{Hirvijoki_2021} and \cite{ZPH2022} an implicit, first order method is presented to conserve mass, momentum, and energy.  Another option is to use the second order implicit midpoint method, which also conserves mass, momentum, and energy exactly.  Conservation of mass is guaranteed regardless of what time integration method is used, as the particle weights remain constant in time.  We now explore the conservation properties for the forward Euler and implicit midpoint methods on \eqref{eq:discrete-in-velocity particle method}.
\begin{proposition}\label{prop:Forward Euler}
    The forward Euler method 
\begin{equation}\label{eq:Forward Euler}
        \frac{1}{\Delta t}\left(v_p^{i,n+1} - v_p^{i,n} \right) 
        =
        - \sum_{j=1}^s\sum_{q=1}^N w_q^j A_{ji}(v_p^{i,n}-v^{j,n}_q)\left(\frac{1}{m_i}F_i^{{\epsilon_i,N}}(v_p^{i,n}) - \frac{1}{m_j}F_j^{\epsilon_j,N}(v_q^{j,n}) \right),
\end{equation}
    conserves total momentum.
\end{proposition}
\begin{proof}
    Multiplying both sides of \eqref{eq:Forward Euler} by $m_iw_p^i$ and summing in $i$ and $p$ gives
    \begin{equation*}
\begin{aligned}
    &\frac{1}{\Delta t} \sum_{i = 1}^{s} \sum_{p = 1}^{N}\left(m_iw_p^iv_p^{i,n+1} - m_iw_p^{i} v_p^{i,n}\right)\\
    & \quad =
    - \sum_{i,j=1}^s\sum_{p,q=1}^N w_p^iw_q^j m_i A_{ji}(v_p^{i,n}-v^{j,n}_q)\left(\frac{1}{m_i}F^{{\epsilon_i,N}}_i(v_p^{i,n}) - \frac{1}{m_j}F_j^{\epsilon_j,N}(v_q^{j,n}) \right)\\
    & \quad =
    \sum_{i,j=1}^s\sum_{p,q=1}^N w_p^iw_q^j m_i A_{ji}(v_p^{i,n}-v^{j,n}_q)\left(\frac{1}{m_i} F_i^{{\epsilon_i,N}}(v_p^{i,n}) - \frac{1}{m_j}F_j^{\epsilon_j,N}(v_q^{j,n}) \right) = 0.\\ 
\end{aligned}
\end{equation*}
\end{proof}
\begin{proposition}
    The implicit midpoint method
\begin{equation}\label{eq:implicit midpoint}
        \frac{1}{\Delta t}\left(v_p^{i,n+1} - v_p^{i,n} \right) 
        =
        - \sum_{j=1}^s\sum_{q=1}^N w_q^j A_{ji}\left(v_p^{i,n+\frac{1}{2}} - v_q^{j,n+\frac{1}{2}} \right)\\
        \left(\frac{1}{m_i}F_i^{{\epsilon_i,N}}\left(v_p^{i,n+\frac{1}{2}}\right) - \frac{1}{m_j}F_j^{\epsilon_j,N}\left(v_q^{j,n+\frac{1}{2}}\right) \right),
\end{equation}
where $v_p^{i,n+\frac{1}{2}} = \frac{1}{2}\left(v_p^{i,n+1} + v_p^{i,n} \right)$, conserves total momentum and energy.
\end{proposition}
\begin{proof}
The proof for total conservation of momentum is nearly identical to the proof in Proposition \ref{prop:Forward Euler}.  For total conservation of energy, dot both sides of \eqref{eq:implicit midpoint} by $2m_iw_p^i v_p^{i,n+\frac{1}{2}}$ and sum in $i$ and $p$
\begin{equation*}
    \begin{aligned}
        &\frac{1}{\Delta t} \sum_{i = 1}^{s} \sum_{p = 1}^{N}m_iw_p^i\left(\left|v_p^{i,n+1}\right|^2 - \left|v_p^{i,n}\right|^2 \right)\\  
        &=
        - 2\sum_{i,j=1}^s\sum_{p,q=1}^N  w_p^iw_q^j m_iv_p^{i,n+\frac{1}{2}} \cdot A_{ji}\left(v_p^{i,n+\frac{1}{2}} - v_q^{j,n+\frac{1}{2}}\right)
        \left(\frac{1}{m_i}F_i^{{\epsilon_i,N}}\left(v_p^{i,n+\frac{1}{2}}\right) - \frac{1}{m_j}F_j^{\epsilon_j,N}\left(v_q^{j,n+\frac{1}{2}}\right) \right)\\
        %
        %
        %
        %
         %
        & =
        2\sum_{i,j=1}^s\sum_{p,q=1}^N w_p^iw_q^jm_i v_q^{j,n+\frac{1}{2}} \cdot A_{ji}\left(v_p^{i,n+\frac{1}{2}}-v_q^{j,n+\frac{1}{2}}\right)
        \left(\frac{1}{m_i}F_i^{{\epsilon_i,N}}\left(v_p^{i,n+\frac{1}{2}}\right) - \frac{1}{m_j}F_j^{\epsilon_j,N}\left(v_q^{j,n+\frac{1}{2}}\right) \right)\\
        &=
        -\sum_{i,j=1}^s\sum_{p,q=1}^N w_p^iw_q^jm_i (v_p^{i,n+\frac{1}{2}} -v_q^{j,n+\frac{1}{2}})\cdot A_{ji}\left(v_p^{i,n+\frac{1}{2}}-v_q^{j,n+\frac{1}{2}}\right)
        \left(\frac{1}{m_i}F_i^{{\epsilon_i,N}}\left(v_p^{i,n+\frac{1}{2}}\right) - \frac{1}{m_j}F_j^{\epsilon_j,N}\left(v_q^{j,n+\frac{1}{2}}\right) \right) = 0.
    \end{aligned}
\end{equation*}
\end{proof}

\section{Numerical examples}
\label{sec:numerical examples}
In this section, we present several numerical examples to validate the theoretical results from Sections \ref{sec:regularized landau} and \ref{sec:Particle method}.  All examples are two dimensional with two species. The computational domain and initialization are chosen according to Subsection \ref{subsec:Initialization and mesh}. The Gaussian mollifier (\ref{eq:gaussian mollifier}) is used in all examples so all the propositions in Section \ref{sec:regularized landau} will hold. The regularization parameter $\epsilon_i$ is chosen as $\epsilon_i = 0.64h_i^{1.98}$.  This choice is motivated by its success in \cite{single_species_Landau_particle_method}. The discrete entropy and macroscopic quantities, whenever needed, are calculated using \eqref{eq:discrete-in-velocity entropy} and \eqref{eq:macroscopic quantities}-\eqref{eq:total macroscopic quantities }.

The first three examples compare the particle solution to a BKW solution, which is an exact solution to \eqref{eq:homogenous Landau equation} in the Maxwell collision case.  We use the name BKW following its counterpart for the multispecies Boltzmann equation \cite{KW1977}. For the multispecies Landau equation, we are not aware of any such solutions existing in the literature. Hence we construct an exact solution from scratch (see Appendix for details). To summarize, assuming the kernel 
\begin{equation}\label{eq:numerical examples kernel}
    A_{ji}(z) = B_{ij}(|z|^2I_d-z\otimes z),
\end{equation}
the BKW solution has the form
\begin{equation}\label{eq:bkw exact solution}
    f_i(t,v) = n_i\left(\frac{m_i}{2\pi K}\right)^{\frac{d}{2}}\exp{\left(-\frac{m_i|v|^2}{2K} \right)}\left(1 - d\frac{1-K}{2K} + \frac{m_i}{K}\frac{1-K}{2K}|v|^2\right), \quad K = 1-C\exp{(-2\beta (d-1)t)},
\end{equation}
where $\sum_{j=1}^2 \frac{B_{ij}}{m_im_j}n_j := \beta_i$ and $\beta_1 = \beta_2 = \beta$. 
 In all three BKW solution examples, we take $C = 1/2$ and $\beta = 1/16$ and $n_1 = n_2 = 1$. The mass ratio and matrix $B_{ij}$ are given at the beginning of each example. The initial time is set as $t = 0$ and the final time is set to $t = 5$.    
 To compare the particle solution to the BKW solution, we use the reconstructed solution (\ref{eq:blob solution}), and define the $L^p$ and $L^{\infty}$ errors as
\begin{equation*}
    \|\tilde{f}^{N}_i - f_i\|^p_{L^p} = \sum_{h_i}h_i^d|\tilde{f}_i^N(v_{h_i}) - f_i(v_{h_i})|^p, \quad \|\tilde{f}^{N}_i - f_i\|_{L^{\infty}} = \max_{h_i}|\tilde{f}_i^N(v_{h_i}) - f_i(v_{h_i})|.
\end{equation*}
The BKW solution examples highlight several important features of the particle method for the multi-species Landau equation \eqref{eq:homogenous Landau equation}.  In Example \ref{Example:BKW Implicit mid} and Example \ref{Example:BKW mass rato 20}, a convergence study validates that the particle method is second order accurate in space.  Example \ref{Example:BKW Implicit mid} compares the effects of using the forward Euler method \eqref{eq:Forward Euler} to the implicit midpoint method \eqref{eq:implicit midpoint} to approximate the system of ODEs \eqref{eq:discrete-in-velocity particle method}. 
 Example \ref{Example:BKW mass rato 20} compares the results of using the same computational domain sizes for each species to the results of using different computational domain sizes for each species.  Specifically, we see a better order of accuracy using different computational domains for each species.  Example \ref{Example:BKW mass ratio 100} shows the particle method's ability to approximate a problem with a large mass ratio.  


The last two examples are Coulomb collision examples, where the kernel is given by
\begin{equation} 
    A_{ji}(z) = B_{ij}\frac{1}{|z|^3}(|z|^2I_d-z\otimes z).
\end{equation}
For both examples, we take the initial condition as 
\begin{equation*}
    f_i(0,v) = n_i\left(\frac{m_i}{2\pi T_i}\right)\exp{\left(-\frac{m_i|v-u_i|^2}{2T_i}\right)},
\end{equation*}
with $n_1 = n_2 = 1$, $u_1 = \left( \frac{1}{2}, \frac{1}{4}\right)^{T}$, $u_2 = \left(-\frac{1}{4},0 \right)^{T}$, $T_1 = \frac{1}{4}$, $T_2 = \frac{1}{8}$. The mass ratios and $B_{ij}$ are given in each example. Since there is no exact solution in this case, we demonstrate the structure-preserving properties of the particle method by examining conservation, entropy decay, and relaxation to a Maxwellian. Therefore, the time length is set long from $t=0$ to $t=50$. In particular, the last example highlights the effects of enforcing the condition $m_1\epsilon_1 = m_2\epsilon_2$ given in Corollary \ref{corollary: Universal temperature relaxation} to ensure a species independent equilibrium temperature versus the effects when this condition is not enforced.  Since $\epsilon_i = 0.64h_i^{1.98}$ and $h_i=L_i/n$ (and the same $n$ is used for both species), the way this condition is enforced is by requiring 
\begin{equation}\label{eq:computational domain constraint}
    \left(\frac{m_1}{m_2}\right)^{1/1.98}L_1 = L_2.
\end{equation}
We choose $L_1$ and $L_2$ satisfying the constraint above and so that the support of the distribution is contained in the computational domain.  Because of this, in both of the Coulomb collision examples, the computational domain is centered around the initial velocities of each species, as opposed to being centered at the origin in order to use a smaller computational domain.


\titleformat{\subsection}
{\normalfont\large\bfseries}{Example~\thesubsection}{1em}{}

\subsection{BKW Example 1}\label{Example:BKW Implicit mid}
In this example, the masses of each species are chosen as $m_1 = 2$ and $m_2 = 1$ and $B_{11} = \frac{1}{8}$, $B_{12} = B_{21} = \frac{1}{16}$, and $B_{22} = \frac{1}{32}$.
The computational domain is $[-3,3]^2$ for species 1 and $[-4,4]^2$ for species 2 (which means $\epsilon_i$ is different for different species).  We obtain the particle solution using different numbers of particles $N = n^2$ with $n = 40,45,50,55,60$. The relative $L^1$, $L^2$, and $L^{\infty}$ errors at the final time $t=5$ are plotted in Figure \ref{fig:Example 1 ROC }, which shows that the particle method developed in this paper is approximately 2nd order accurate in space (w.r.t. $h_i$, the initial mesh size). Figure \ref{fig:Example 1 total energy and entropy} shows the time evolution of the total energy and total entropy for $n^2 = 50^2$ particles with respect to different time steps $\Delta t$.

Here we advance the particle method using two time integrators: forward Euler and implicit midpoint as discussed in Subsection \ref{subsec:time}. For the implicit point method, we use the fixed point iteration at each inner time step with a tolerance of $10^{-8}$ for convergence. We also tracked the time step needed (they may not be optimal but often the case a larger time step would result in the convergence criterion not satisfied). From the numerical results, we can conclude the following: 1) For the typical particle numbers we tested, the error from particle approximation still dominates so it makes little difference of using either first or second order time integrator in terms of accuracy. 2) The implicit midpoint can preserve the energy up to a small error that is dominated by the choice of tolerance in the fixed point iteration, while the forward Euler can preserve the energy up to $O(\Delta t)$. However, the implicit midpoint method often requires smaller time step in order to guarantee the convergence which makes it more expensive than the forward Euler method. More sophisticated iteration schemes may help on convergence and we leave it for future studies. 

Based on the above observations, we choose to use the forward Euler method for the rest of numerical examples.

\newcommand{\ExampleOneRateofConvergenceSpeciesOne}{\includegraphics[width=.49\textwidth]{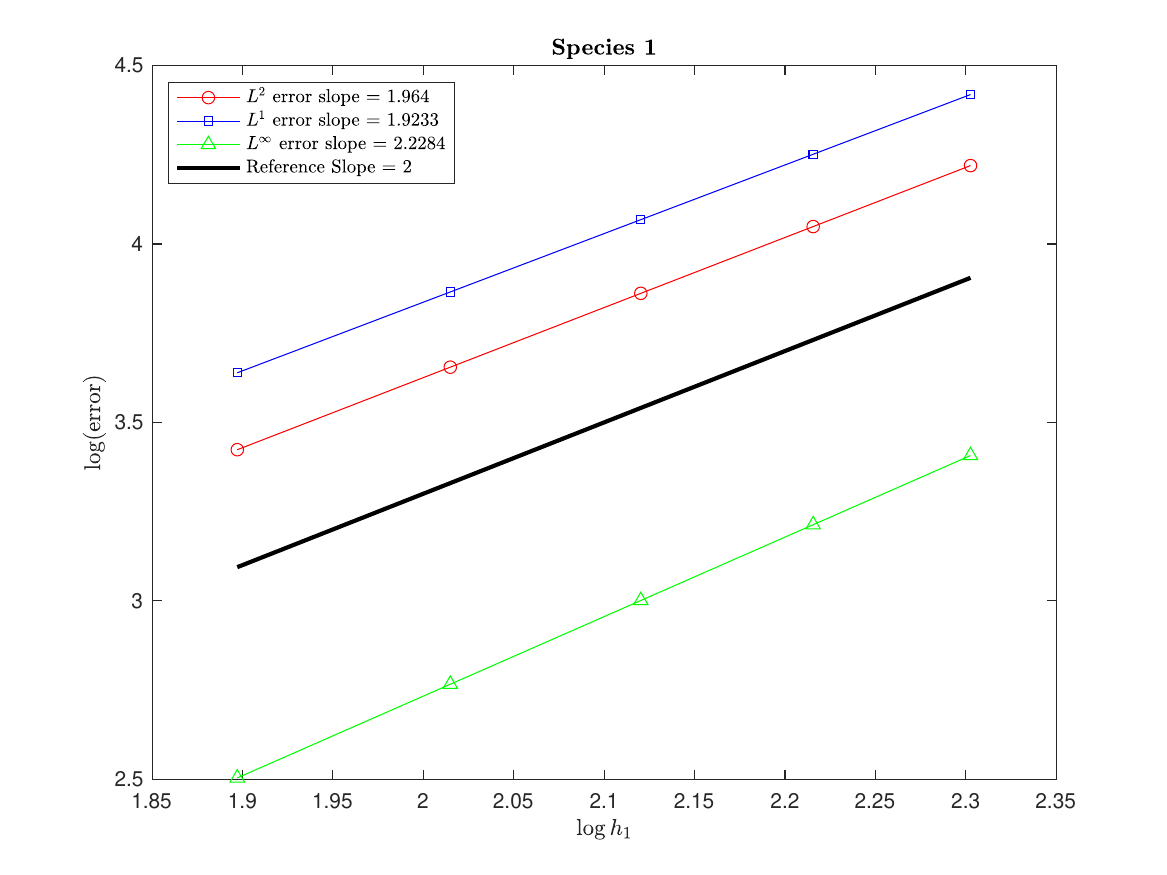}}

\newcommand{\ExampleOneRateofConvergenceSpeciesTwo}{\includegraphics[width=.49\textwidth]{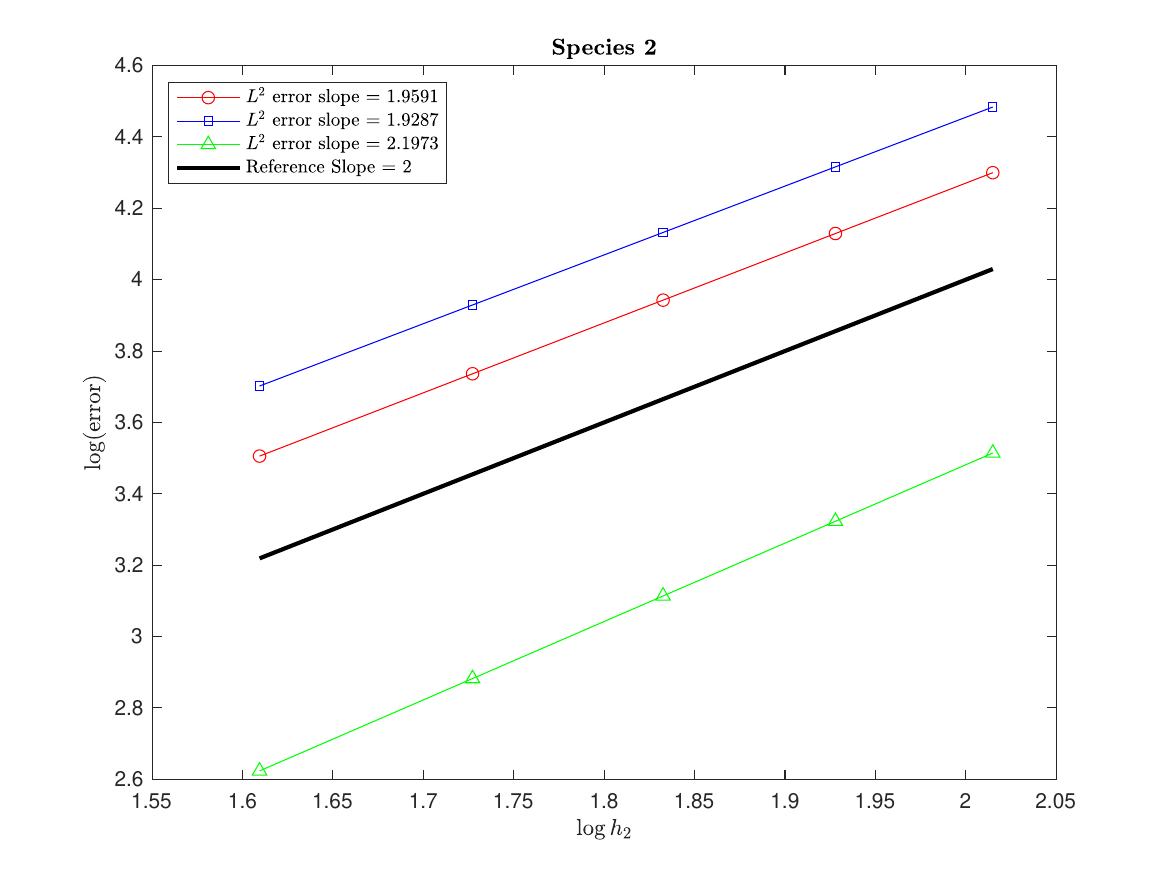}}

\newcommand{\ExampleOneBRateofConvergenceSpeciesOne}{\includegraphics[width=.49\textwidth]{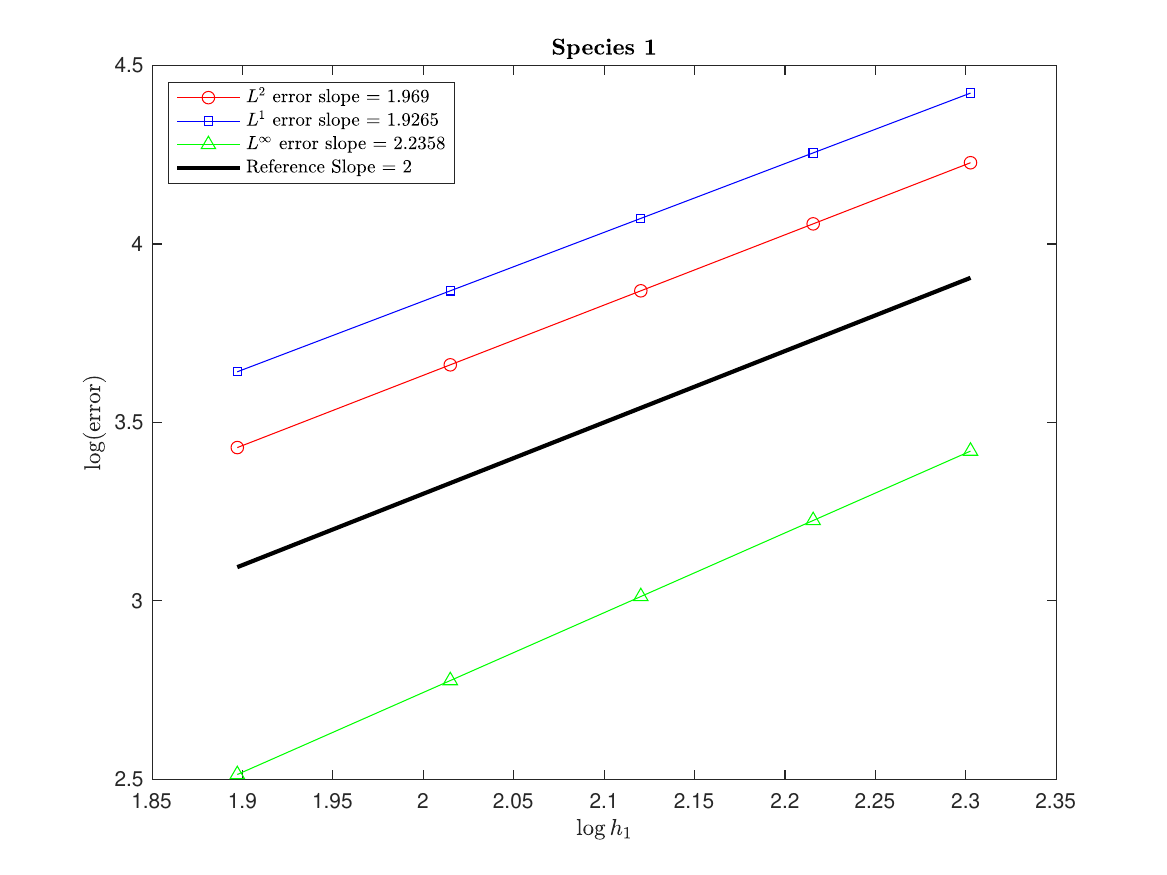}}

\newcommand{\ExampleOneBRateofConvergenceSpeciesTwo}{\includegraphics[width=.49\textwidth]{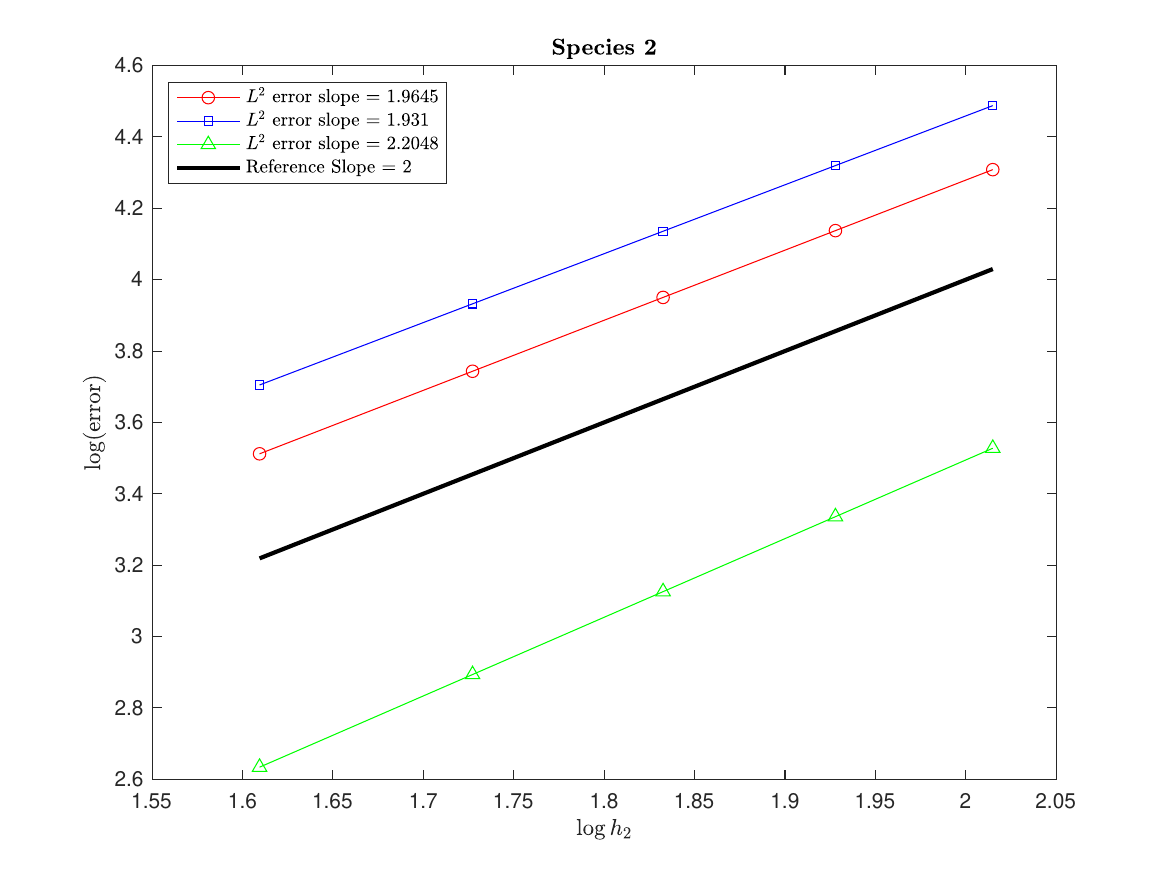}}


\newcommand{\ExampleOneMomentumVone}{\includegraphics[width=.49\textwidth]{Example_1_Total_Momentum_v1 direction_dt_test.eps}}

\newcommand{\ExampleOneMomentumVtwo}{\includegraphics[width=.49\textwidth]{Example_1_Total_Momentum_v2_direction_dt_test.eps}}

\newcommand{\ExampleOneBMomentumVone}{\includegraphics[width=.49\textwidth]{Example_1b_Total_Momentum_v1_direction_dt_test.eps}}

\newcommand{\ExampleOneBMomentumVtwo}{\includegraphics[width=.49\textwidth]{Example_1b_Total_Momentum_v2_direction_dt_test.eps}}


\newcommand{\ExampleOneEnergy}{\includegraphics[width=.49\textwidth]{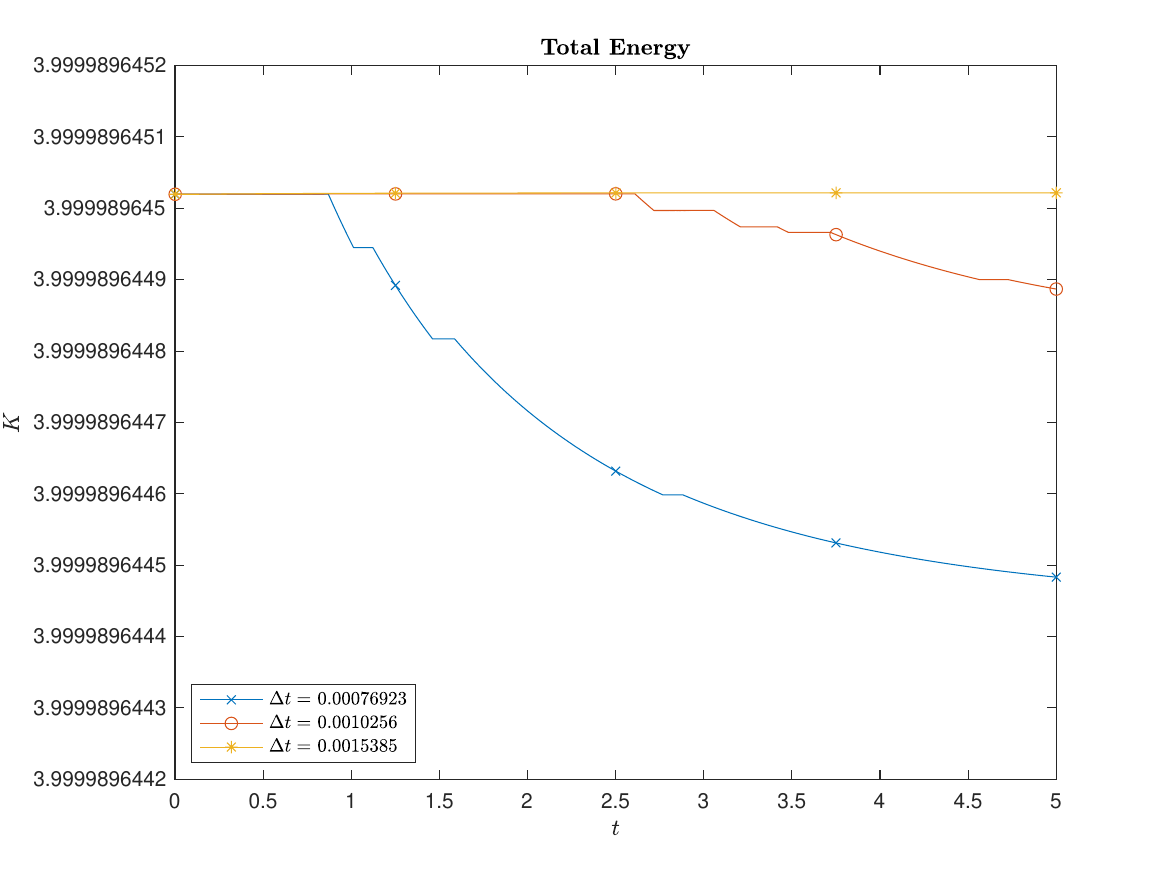}}

\newcommand{\ExampleOneEntropy}{\includegraphics[width=.49\textwidth]{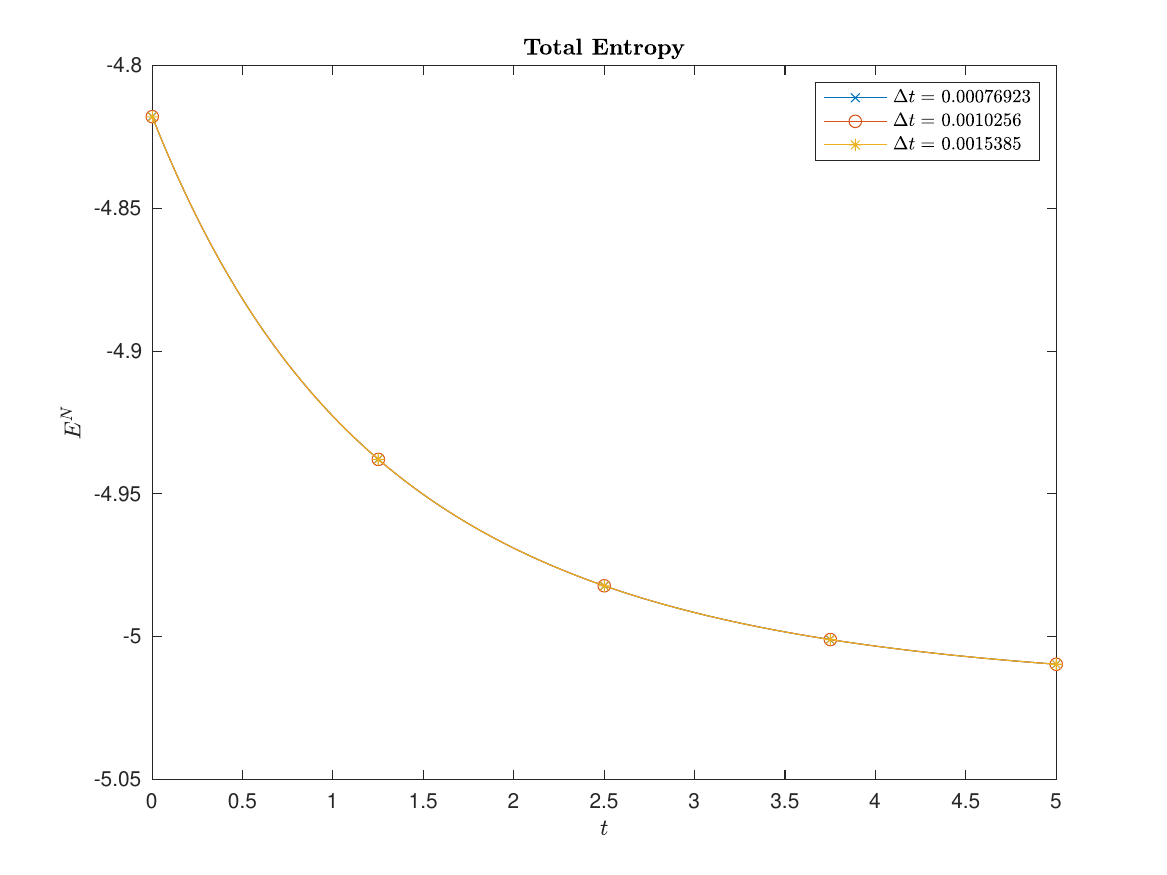}}

\newcommand{\ExampleOneBEnergy}{\includegraphics[width=.49\textwidth]{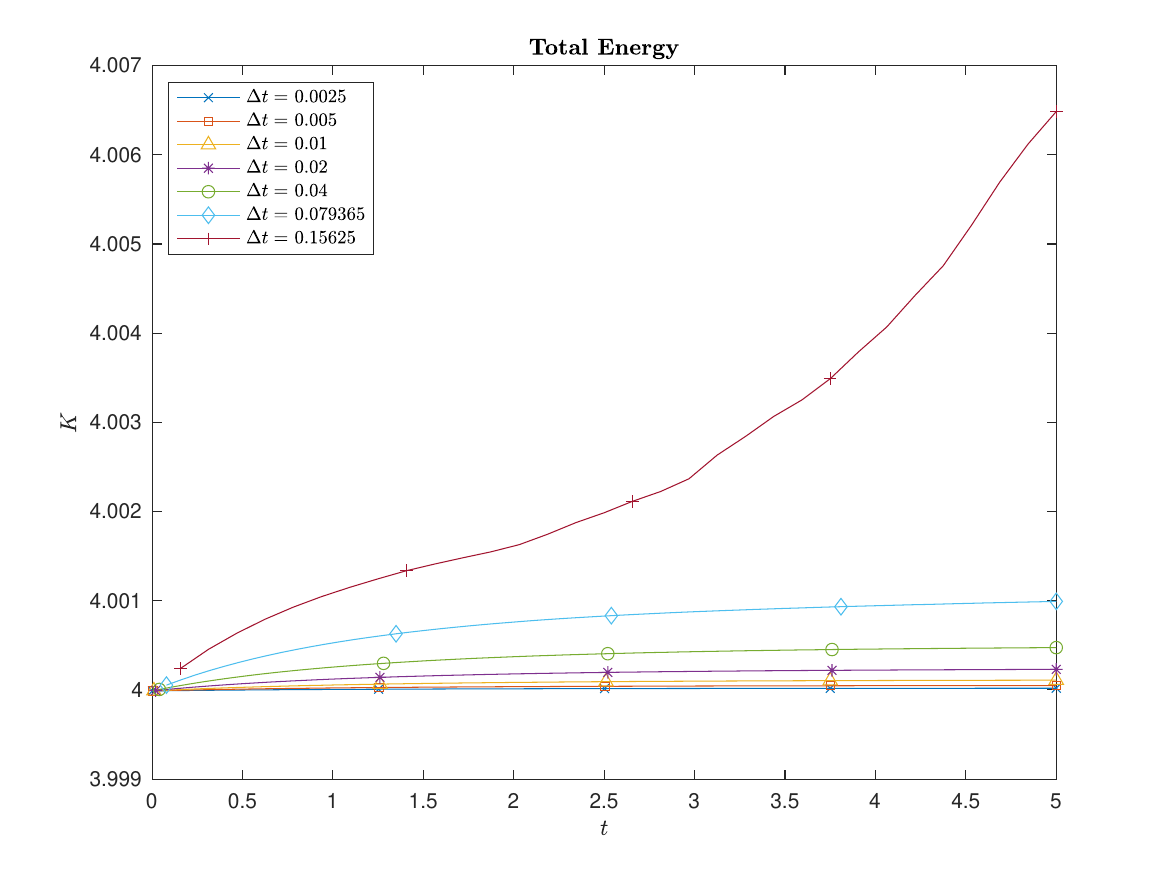}}

\newcommand{\ExampleOneBEntropy}{\includegraphics[width=.49\textwidth]{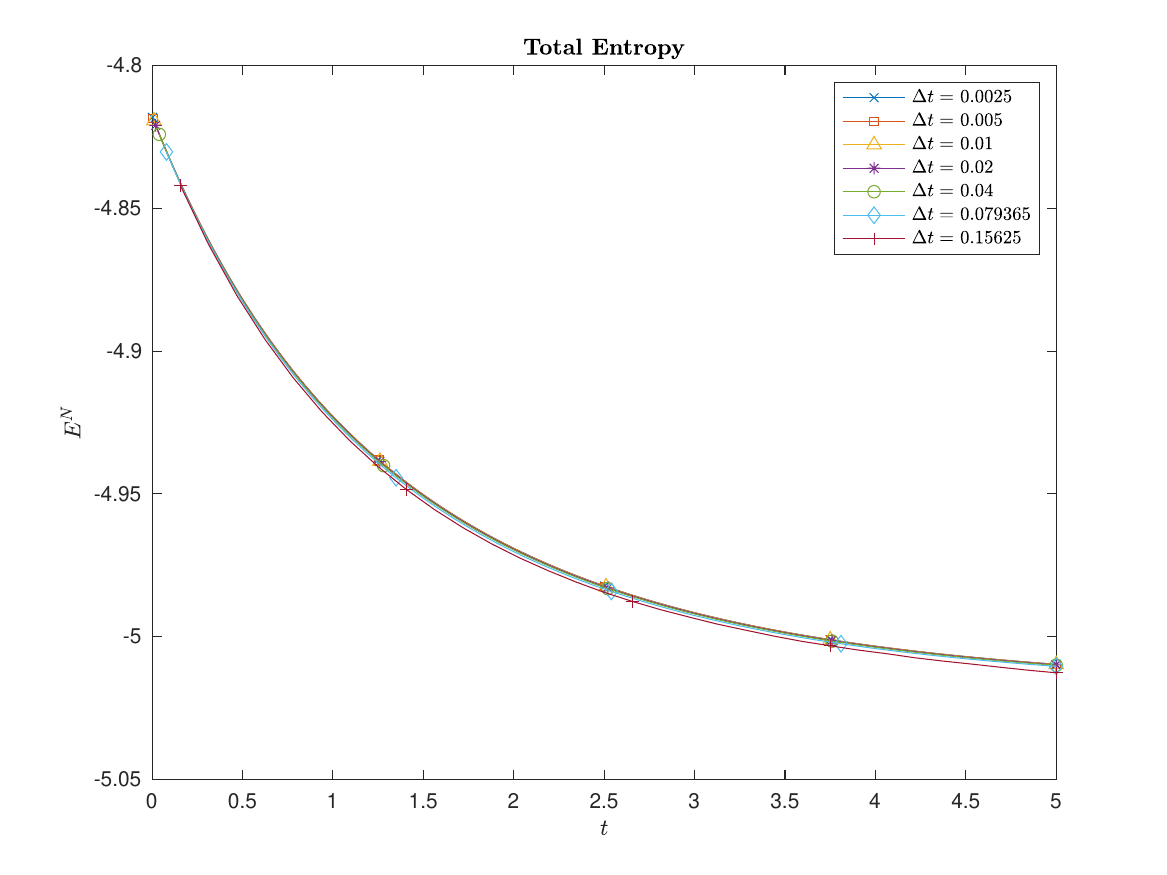}}


\begin{figure}[htp]
    \centering
        \ExampleOneRateofConvergenceSpeciesOne
        \ExampleOneBRateofConvergenceSpeciesOne \\
        \ExampleOneRateofConvergenceSpeciesTwo
        \ExampleOneBRateofConvergenceSpeciesTwo\\
        \begin{tabular}{|c|c|c|c|c|c|}
        \hline
        $n$ & 40 & 45 & 50 & 55 & 60\\
        \hline
        $\Delta t$ & 0.0025 & 0.002 & 0.00153 & 0.00125 & 0.001 \\ 
        \hline
        \end{tabular}
        \begin{tabular}{|c|c|c|c|c|c|}
        \hline
        $n$ & 40 & 45 & 50 & 55 & 60\\
        \hline
        $\Delta t$ & 0.01 & 0.01 & 0.01 & 0.01 & 0.005 \\ 
        \hline
        \end{tabular}
        \caption{Example \ref{Example:BKW Implicit mid}:  Relative $L^{\infty}$, $L^1$, and $L^2$ norms of the error at time $t = 5$ with respect to different $h$ for species 1 and species 2.  The plots on the left were produced using the implicit midpoint method \eqref{eq:implicit midpoint} to approximate $\eqref{eq:discrete-in-velocity particle method}$ and the plots on the right were produced using the forward Euler method \eqref{eq:Forward Euler} to approximate \eqref{eq:discrete-in-velocity particle method}.  The tables below the plots show the time step $\Delta t$ used for each value of $n$ for the implicit midpoint method (left) and the forward Euler method (right).
        }
        \label{fig:Example 1 ROC }
\end{figure}

\begin{figure}
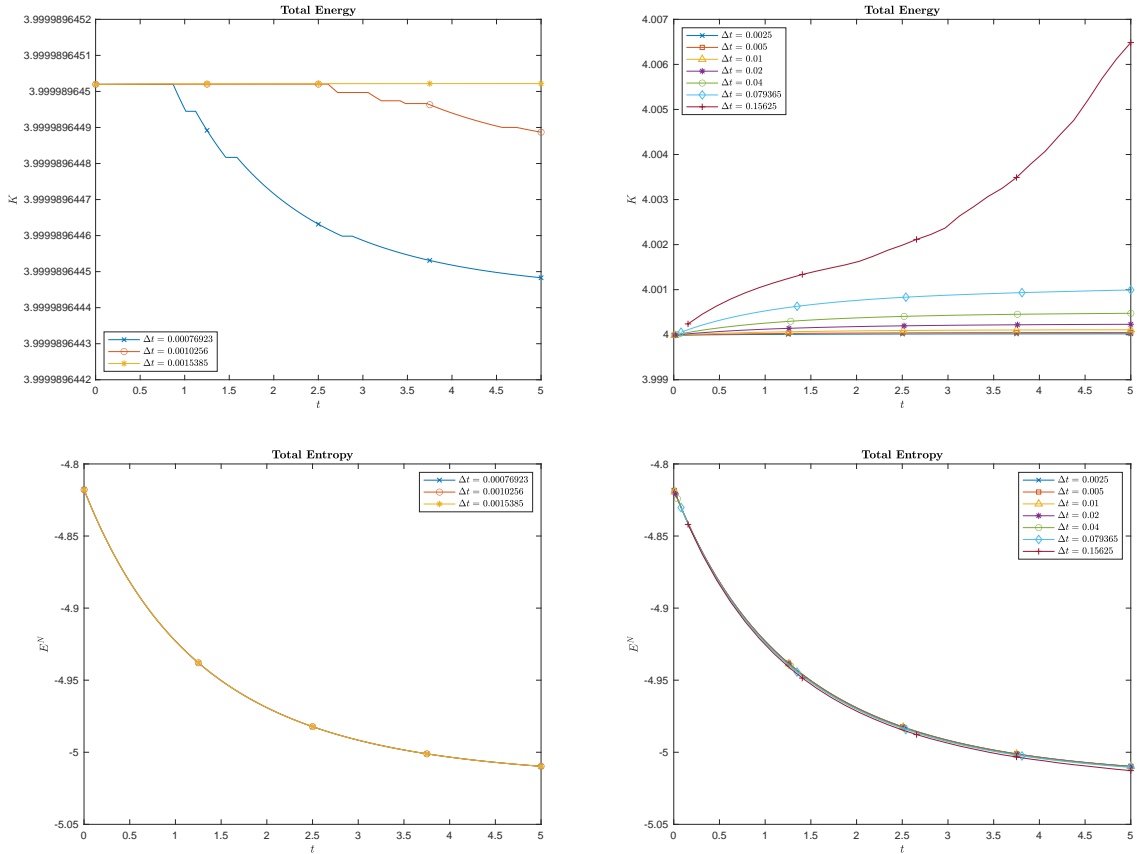

    \centering
    \ExampleOneEnergy
    \ExampleOneBEnergy \\
    \ExampleOneEntropy
    \ExampleOneBEntropy
    \caption{Example \ref{Example:BKW Implicit mid}:  Time evolution of the total energy (top row) and total entropy (bottom row). 
 The plots on the left were produced using the implicit midpoint method \eqref{eq:implicit midpoint} to approximate \eqref{eq:discrete-in-velocity particle method} and the plots on the right were produced using the forward Euler method \eqref{eq:Forward Euler} to approximate \eqref{eq:discrete-in-velocity particle method}. }
    \label{fig:Example 1 total energy and entropy}
\end{figure}

\subsection{BKW Example 2}\label{Example:BKW mass rato 20}
In this example, the masses of each species are chosen as $m_1 = 20$ and $m_2 = 1$ and $B_{11} = \frac{1}{2}$, $B_{12} = B_{21} = \frac{49}{40}$, and $B_{22} = \frac{1}{800}$.  With a mass ratio of $20$, it is important to use different domain size (hence different regularization parameter $\epsilon_i$) for each species. To illustrate this, we choose the computational domain $[-0.9,0.9]^2$ for species 1 and $[-4,4]^2$ for species 2, and compare the results with those using the same computational domain $[-4,4]^2$ for both species.  

Using $n^2 = 60^2, 80^2, 100^2, 120^2$ particles and a time step $\Delta t = 0.001$, Figure \ref{fig:Example 2 ROC} shows that when a smaller computational domain is used for the heavier species, the particle method is approximately 2nd order accurate while using the same computational domain for each species, the order of accuracy degrades.  In Figure \ref{fig:Example 2 L2 error} time evolution of the relative $L^2$ error also confirms better accuracy when using a smaller computational domain for the heavier species.  

\newcommand{\ExampleTwoRateofConvergenceSpeciesOne}{\includegraphics[width=.49\textwidth]{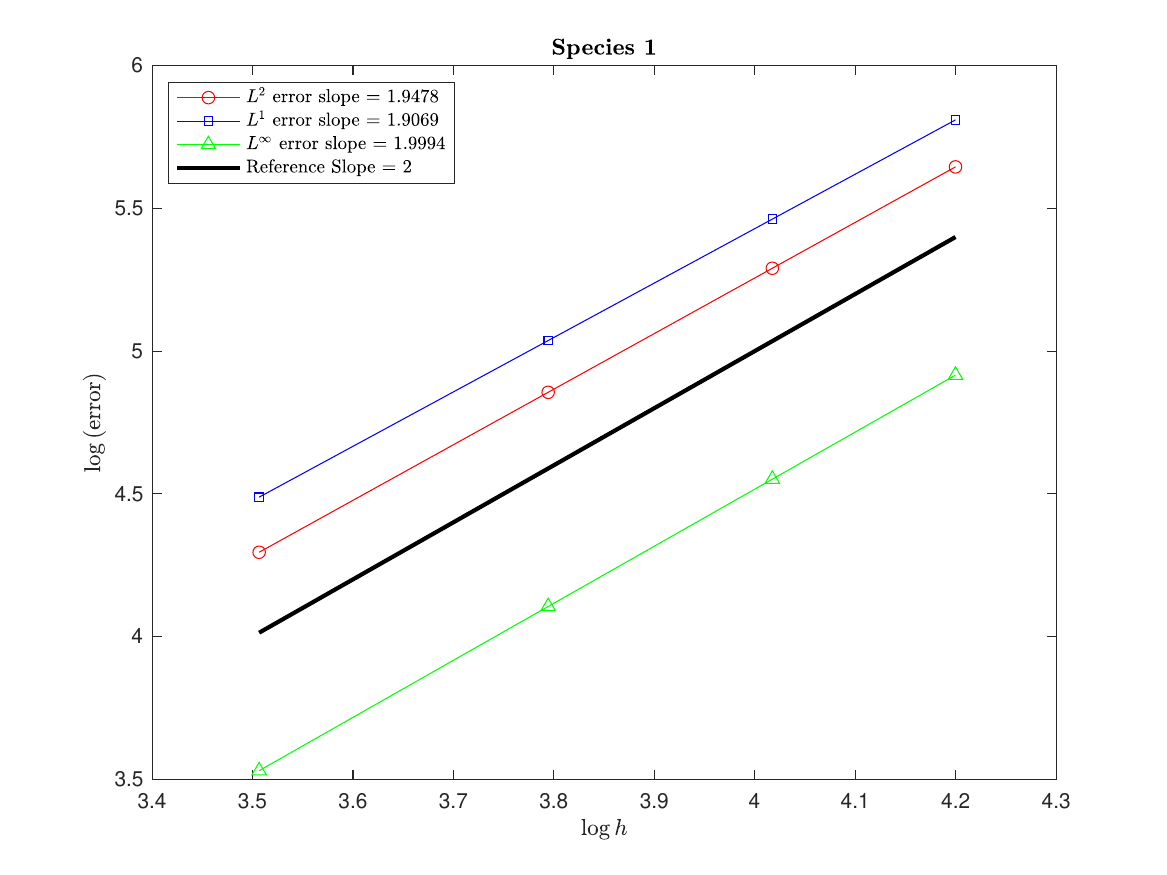}}

\newcommand{\ExampleTwoRateofConvergenceSpeciesTwo}{\includegraphics[width=.49\textwidth]{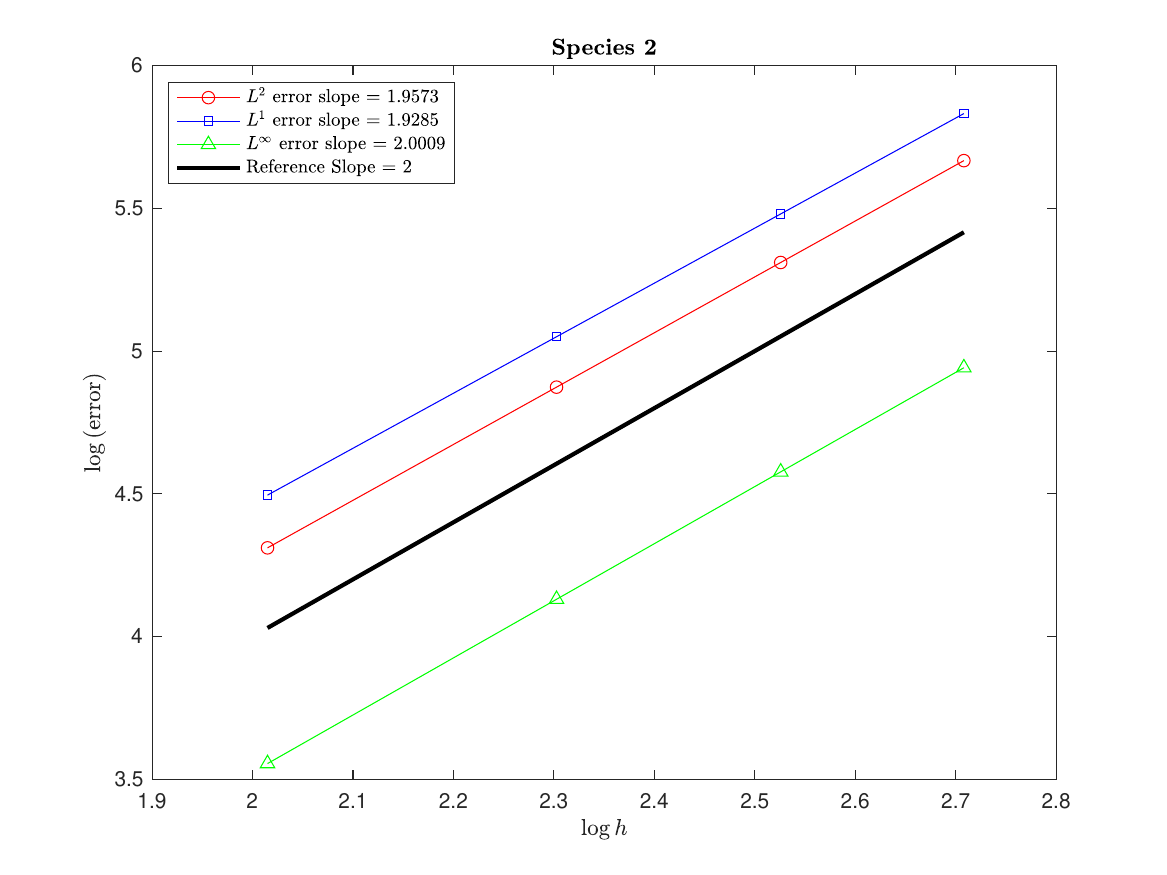}}

\newcommand{\ExampleTwoRateofConvergenceSpeciesOneSameComputationalDomain}{\includegraphics[width=.49\textwidth]{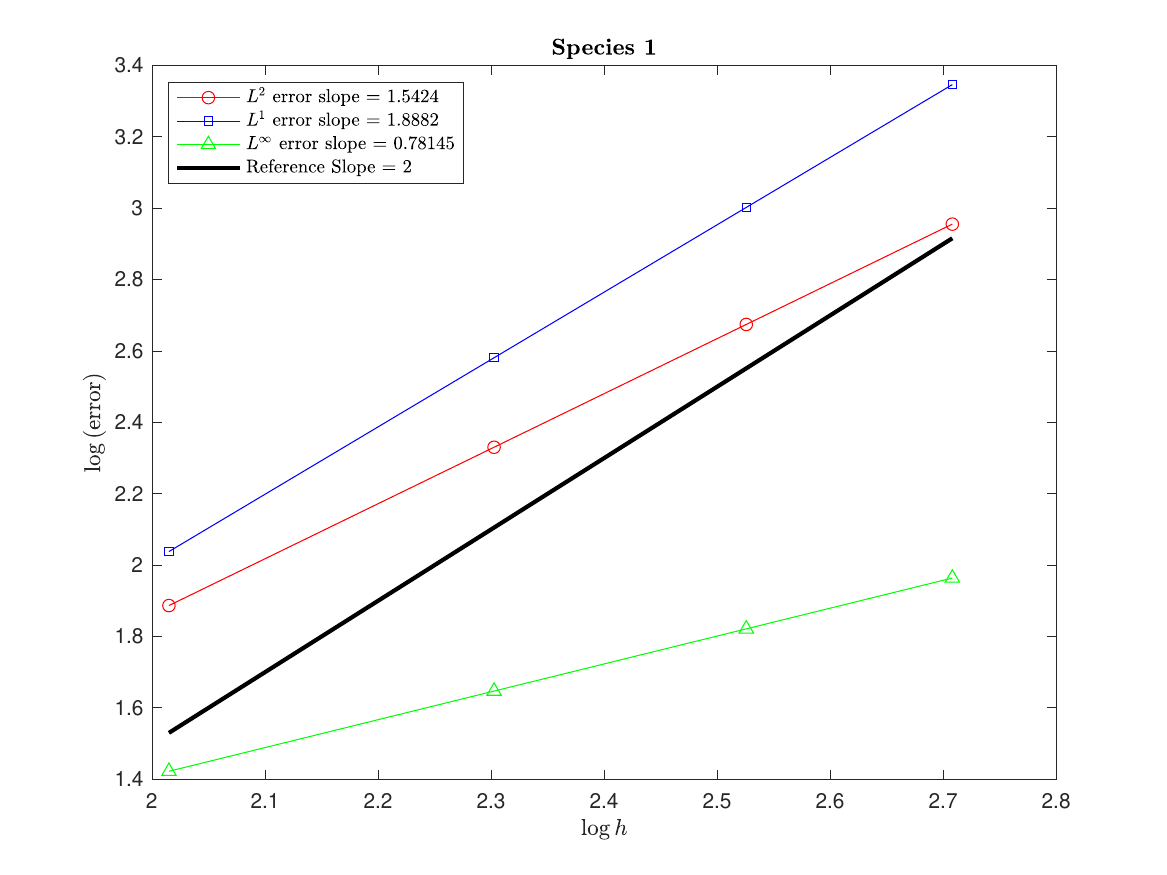}}

\newcommand{\ExampleTwoRateofConvergenceSpeciesTwoSameComputationalDomain}{\includegraphics[width=.49\textwidth]{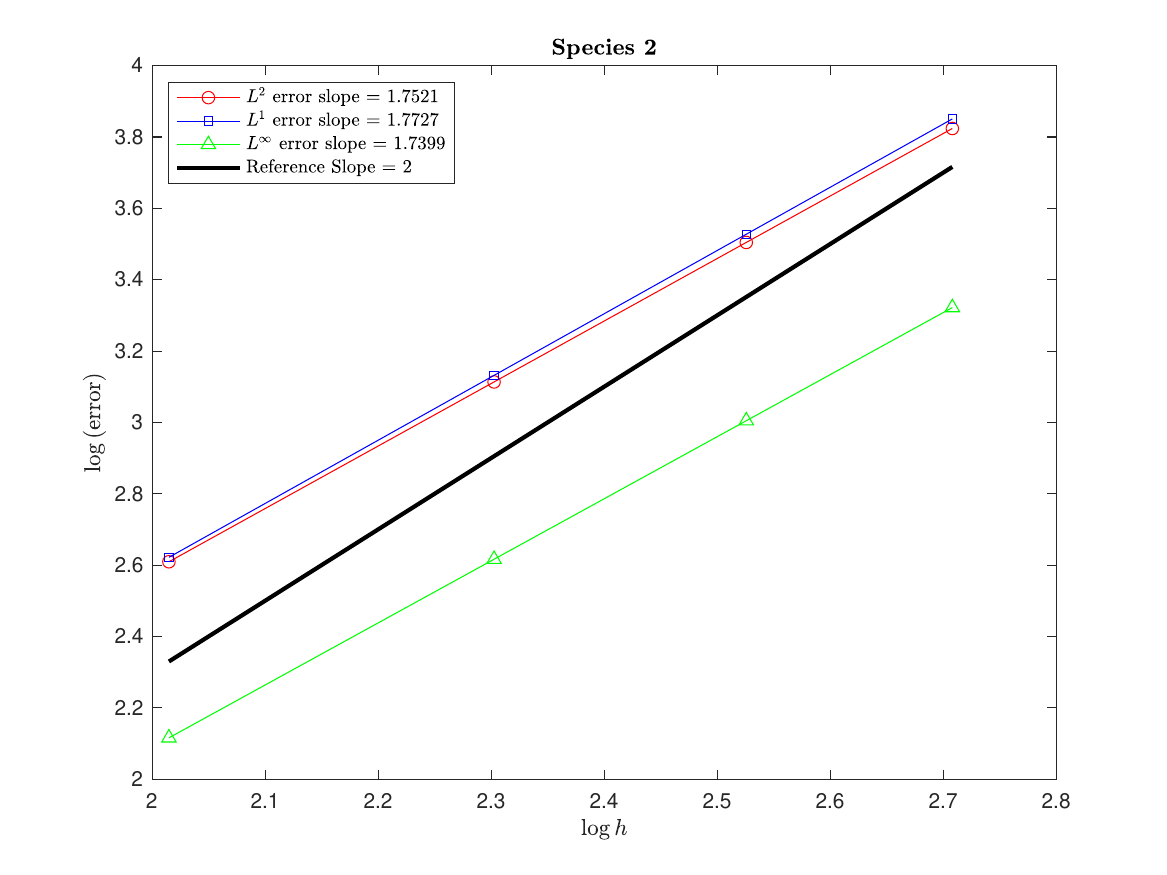}}

\begin{figure}[htp]
    \begin{center}
        \ExampleTwoRateofConvergenceSpeciesOne
        \ExampleTwoRateofConvergenceSpeciesTwo \\
        \ExampleTwoRateofConvergenceSpeciesOneSameComputationalDomain
        \ExampleTwoRateofConvergenceSpeciesTwoSameComputationalDomain
        \caption{Example \ref{Example:BKW mass rato 20}:  Relative $L^{\infty}$, $L^1$, and $L^2$ norms of the error at time $t = 5$ with respect to different $h$.  The upper left and right are plots when the computational domain is $[-0.9,0.9]^2$ for species 1 and $[-4,4]^2$ for species 2.  The lower left and right are plots when the computational domain is $[-4,4]^2$ for both species.}
        \label{fig:Example 2 ROC}
    \end{center}
\end{figure}



\newcommand{\ExampleTwoLTwoErrorSpeciesOneDifferentCompDomain}{\includegraphics[width=.49\textwidth]{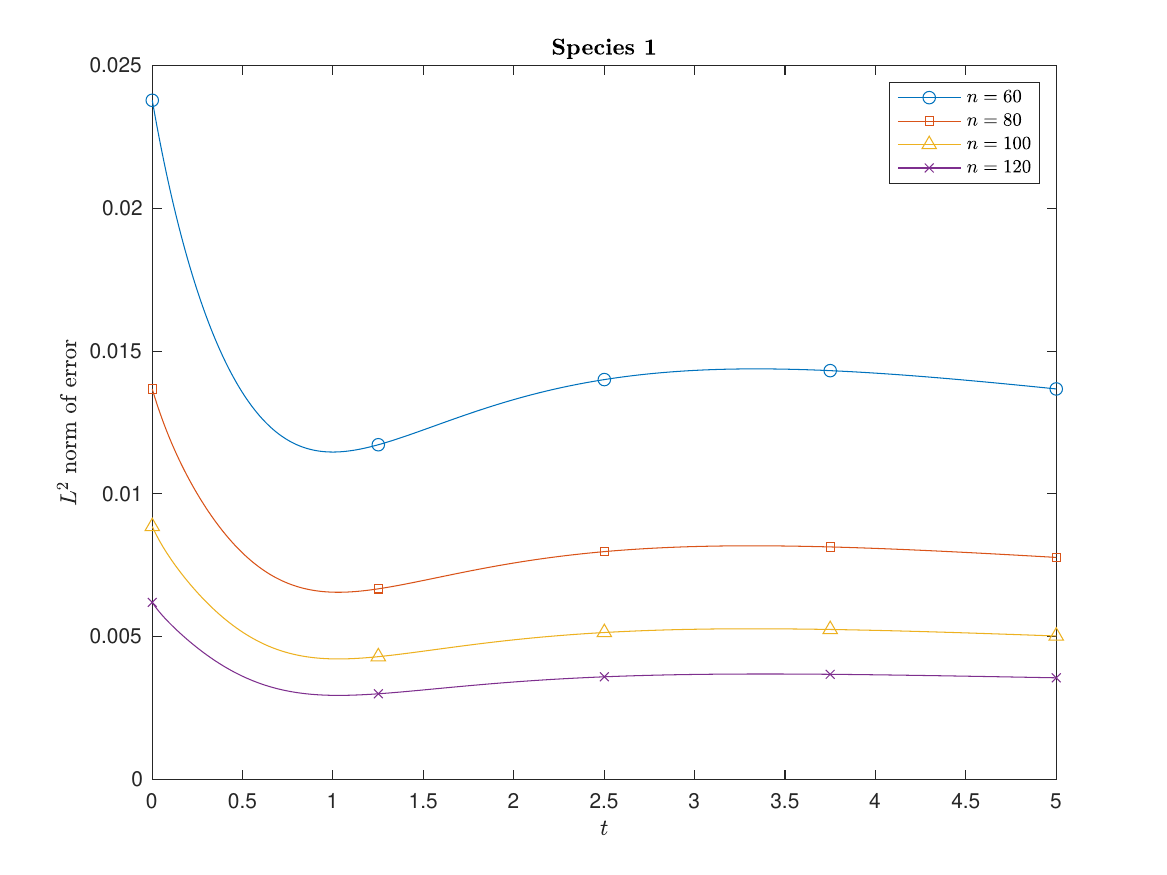}}

\newcommand{\ExampleTwoLTwoErrorSpeciesTwoDifferentCompDomain}{\includegraphics[width=.49\textwidth]{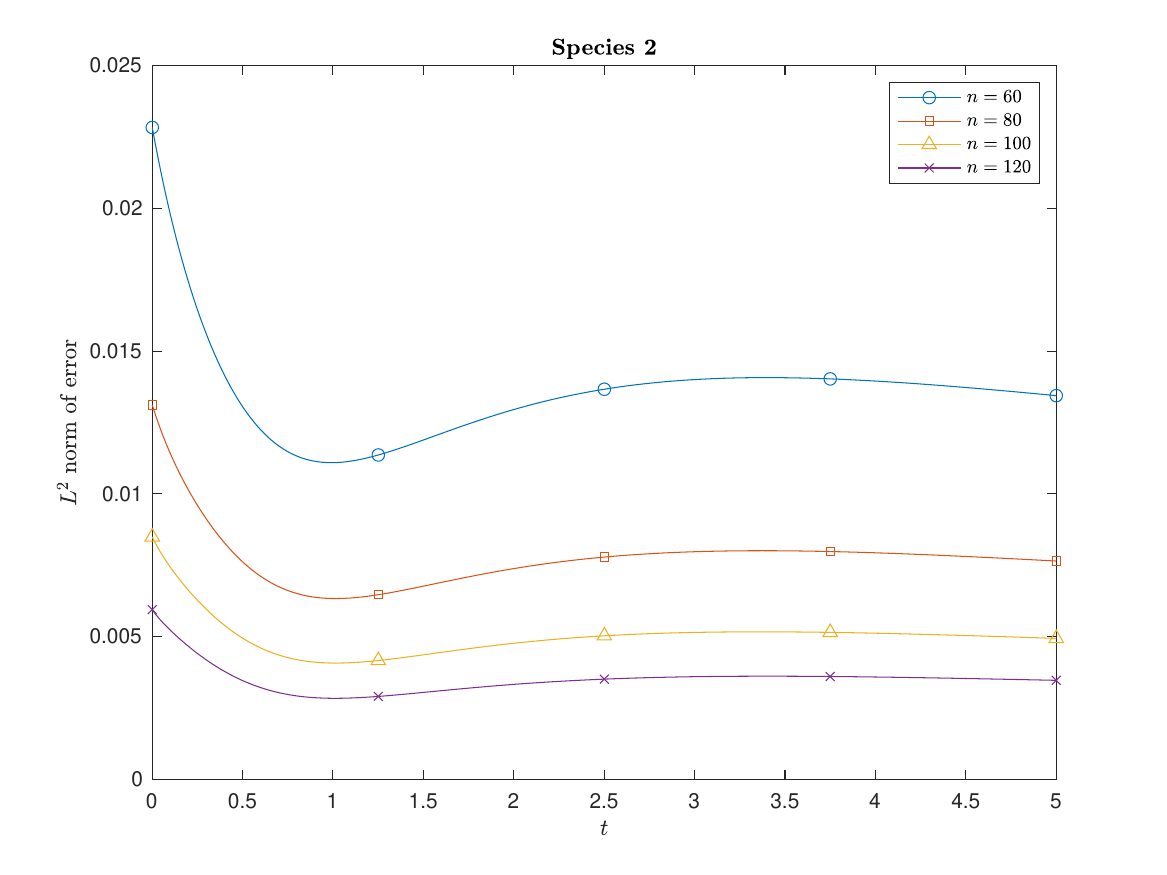}}

\newcommand{\ExampleTwoLTwoErrorSpeciesOneSameCompDomain}{\includegraphics[width=.49\textwidth]{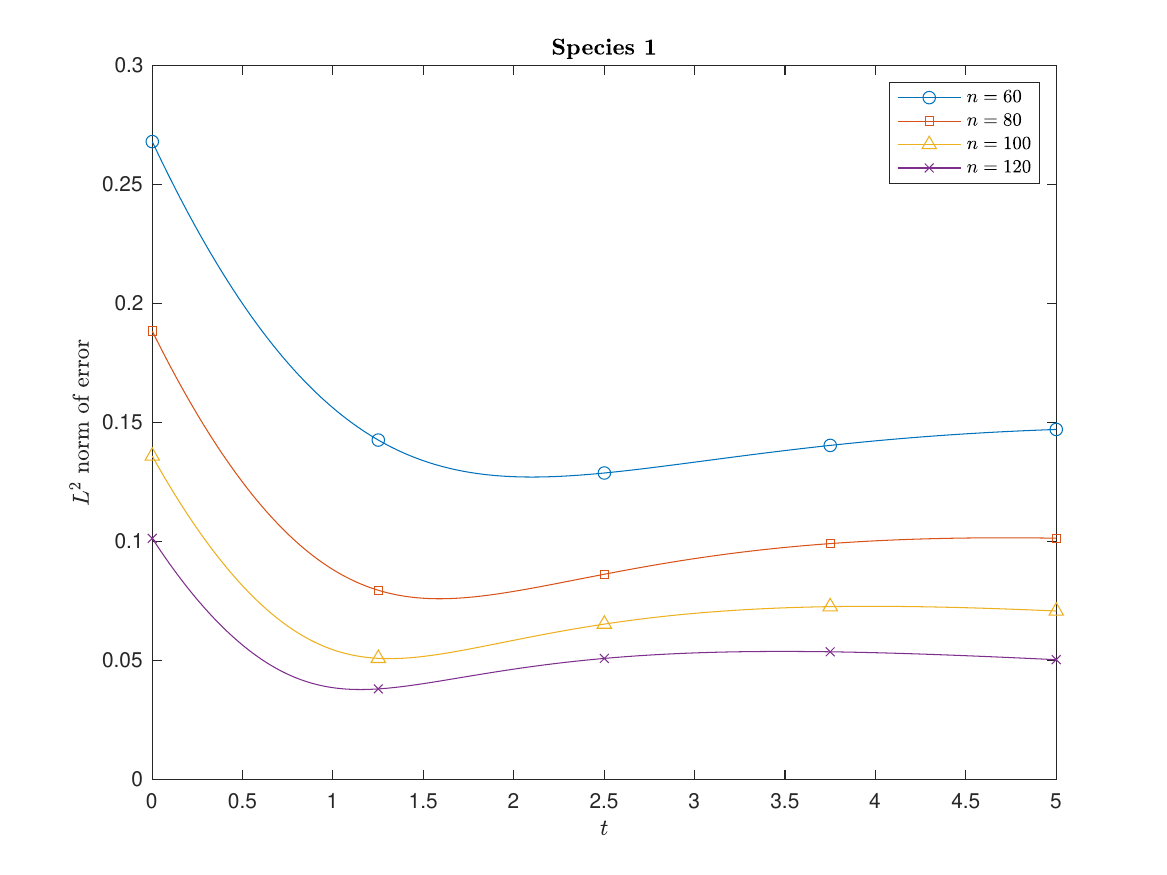}}

\newcommand{\ExampleTwoLTwoErrorSpeciesTwoSameCompDomain}{\includegraphics[width=.49\textwidth]{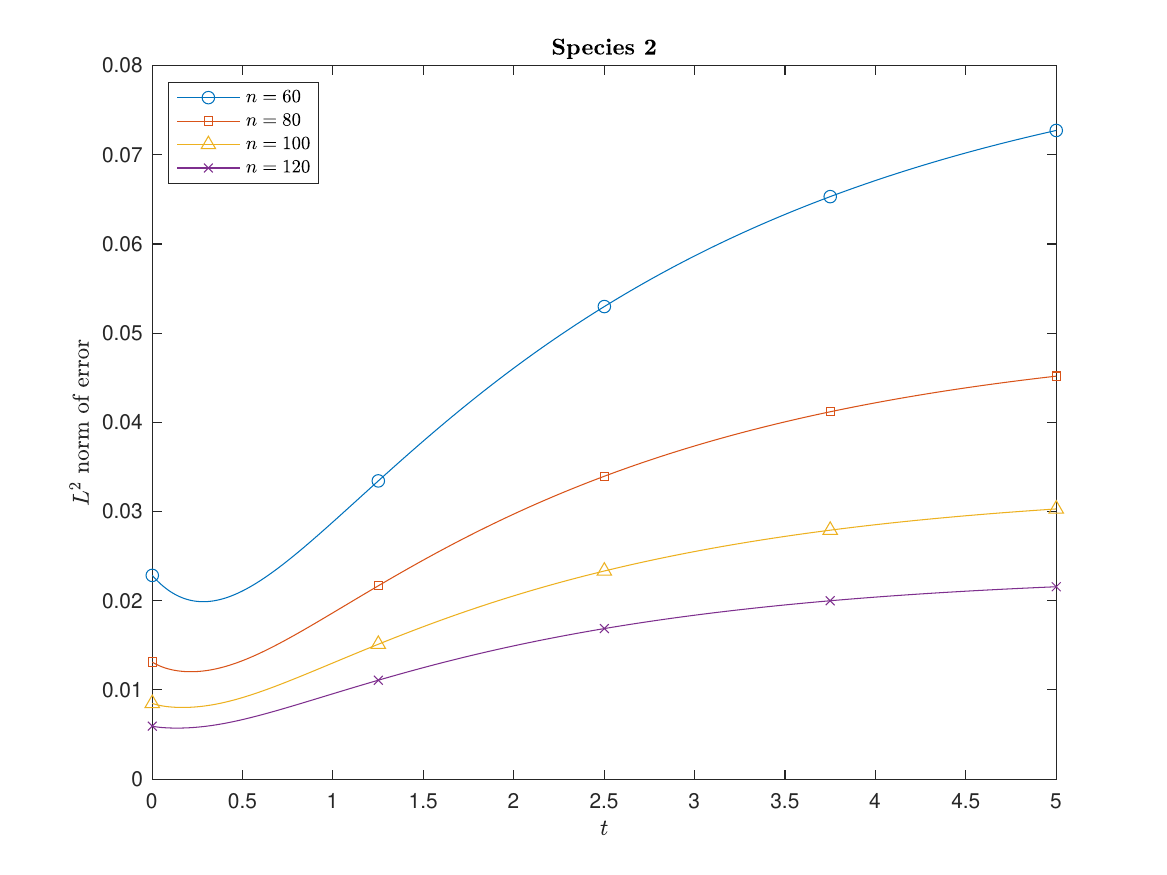}}

\begin{figure}[htp]
    \begin{center}
        \ExampleTwoLTwoErrorSpeciesOneDifferentCompDomain
        \ExampleTwoLTwoErrorSpeciesTwoDifferentCompDomain \\
        \ExampleTwoLTwoErrorSpeciesOneSameCompDomain
        \ExampleTwoLTwoErrorSpeciesTwoSameCompDomain
        \caption{Example \ref{Example:BKW mass rato 20}:  Time evolution of the relative $L^2$ error with respect to different number of particles.  The upper left and right plots show the relative error evolution when the computational domain is $[-0.9,0.9]^2$ for species 1 and $[-4,4]^2$ for species 2.  The lower left and right plots show the relative error evolution when the computational domain is $[-4,4]^2$ for both species.}
        \label{fig:Example 2 L2 error}
    \end{center}
\end{figure}

\subsection{BKW Example 3}\label{Example:BKW mass ratio 100}
In Examples \ref{Example:BKW Implicit mid} and \ref{Example:BKW mass rato 20}, the BKW solutions have mass ratios $m_1$/$m_2 = 2$ and $m_1/m_2 = 20$, respectively.  In reality, the mass ratio of two different species in a plasma can be much larger, and therefore in this example we consider $m_1 = 100$ and $m_2 = 1$ and $B_{11} = \frac{1}{2}$, $B_{12} = B_{21} = \frac{1249}{200}$, $B_{22} = \frac{1}{20000}$.  The computational domain is $[-0.4,0.4]^2$ for species 1 and $[-4,4]^2$ for species 2. We use $n^2 = 100^2$ particles and a time step of $\Delta t = 10^{-4}$. Figure \ref{fig:Example 3 total energy and entropy and PDFs} shows the time evolution of the total energy and total entropy, as well as cross-sections of the numerical solutions at time $t = 5$. The solutions are captured well, along with the expected energy conservation and entropy decay properties.

\newcommand{\ExampleThreeTotalEnergy}{\includegraphics[width=.49\textwidth]{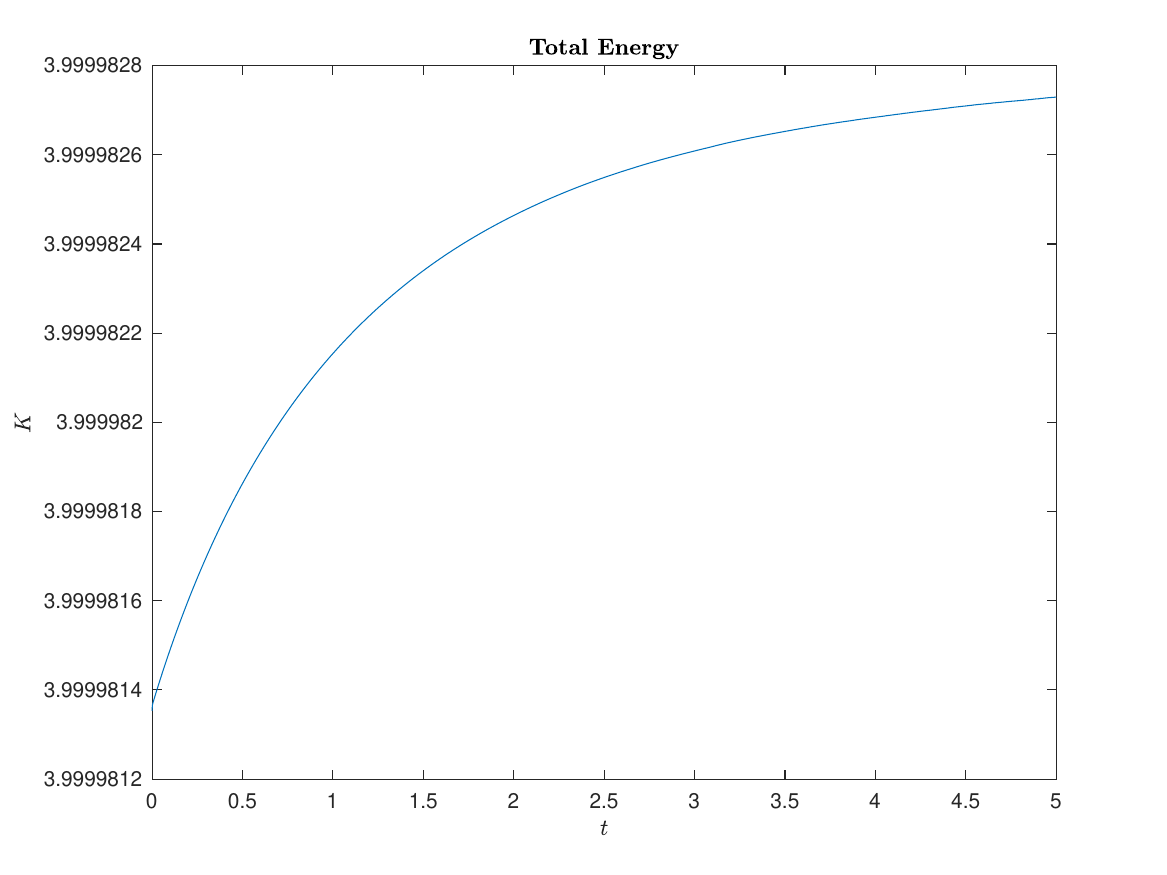}}

\newcommand{\ExampleThreeTotalEntropy}{\includegraphics[width=.49\textwidth]{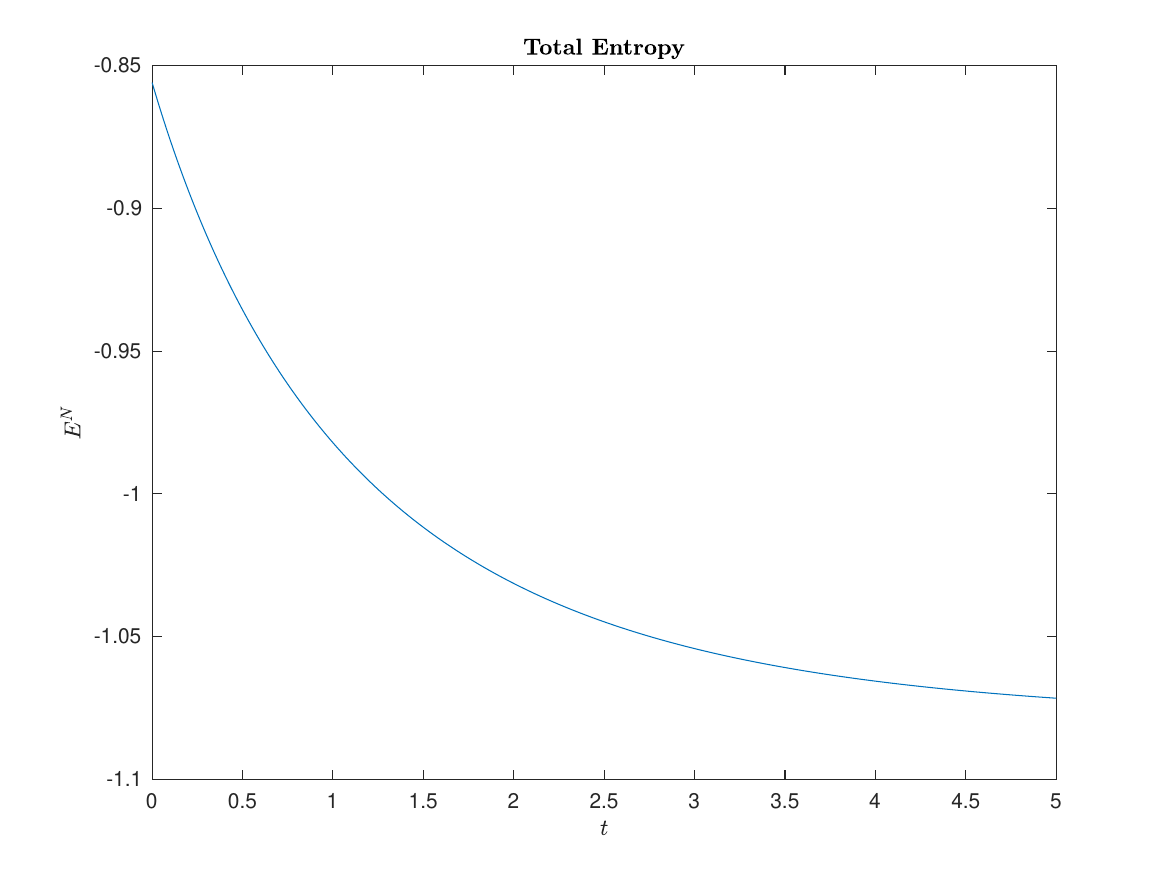}}

\newcommand{\ExampleThreeSpeciesOnePDF}{\includegraphics[width=.49\textwidth]{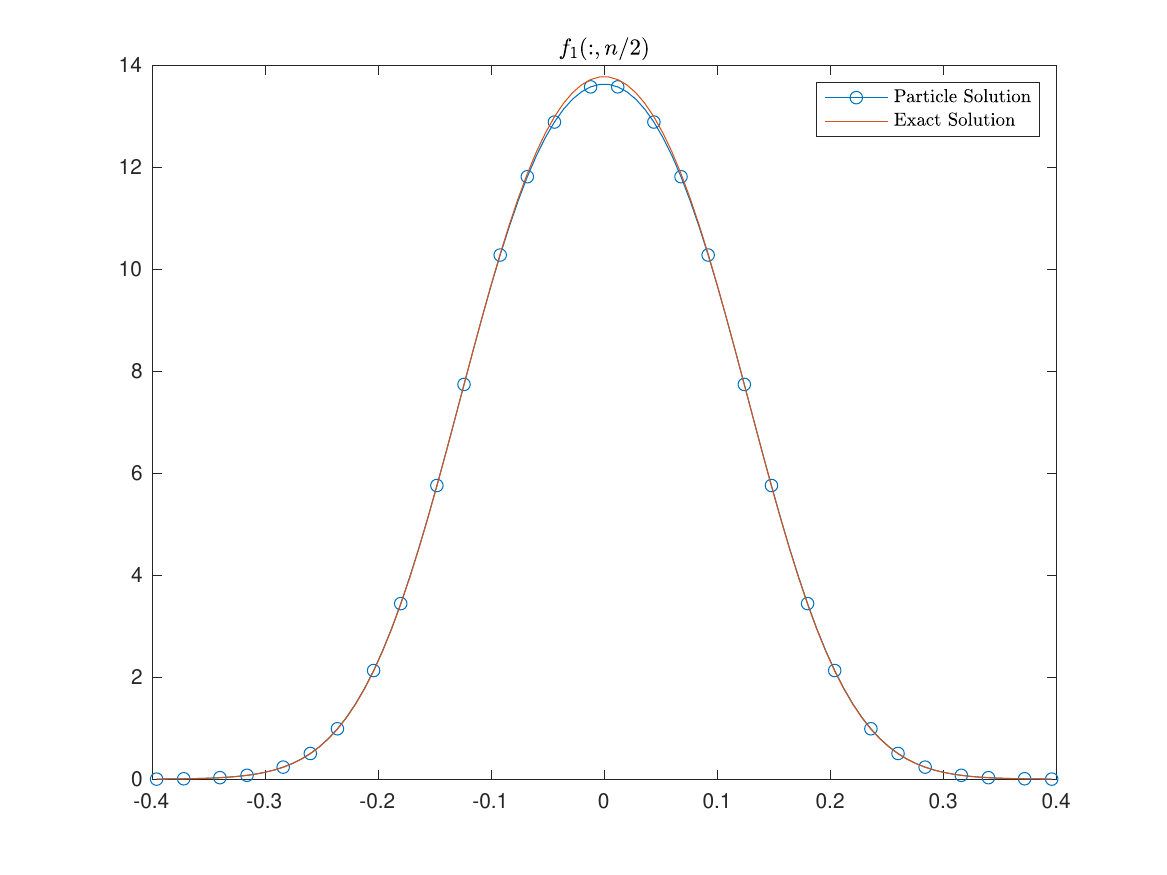}}

\newcommand{\ExampleThreeSpeciesTwoPDF}{\includegraphics[width=.49\textwidth]{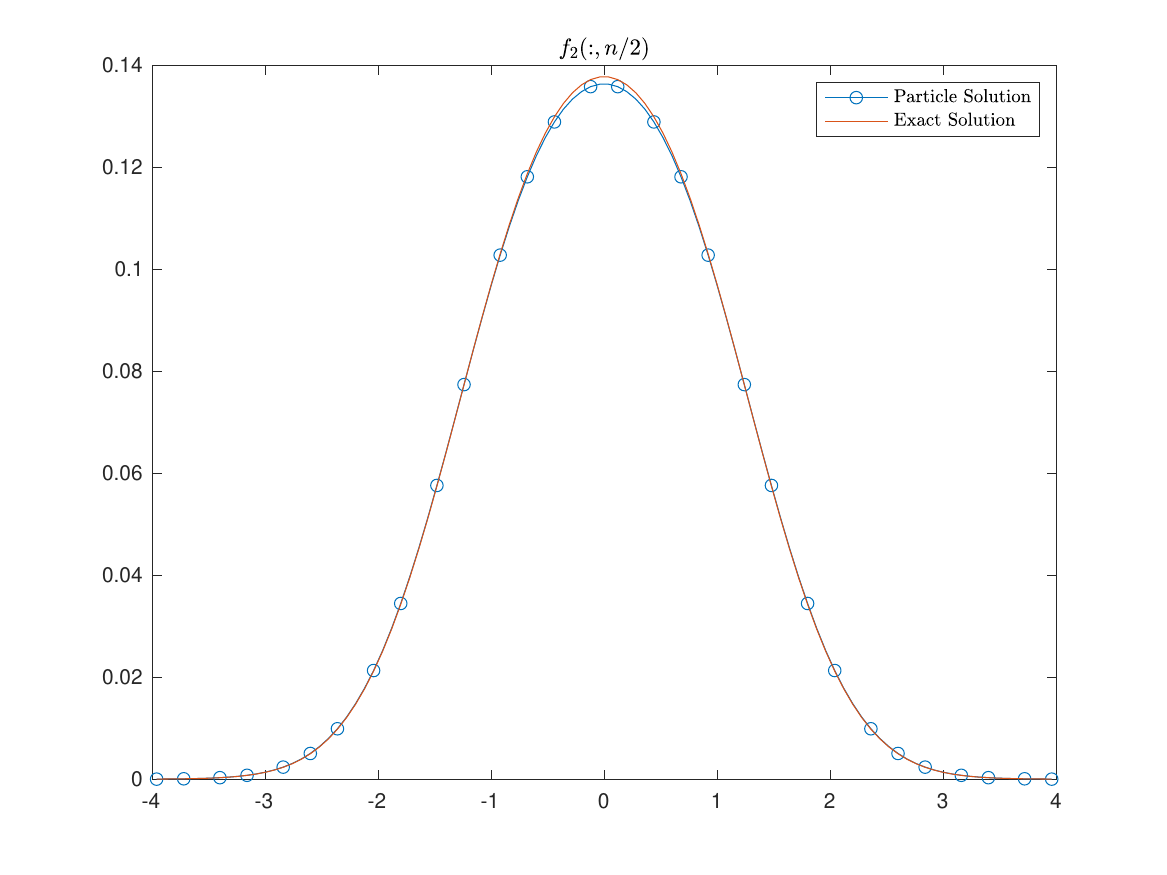}}

\begin{figure}[htp]
    \centering
    \ExampleThreeTotalEnergy
    \ExampleThreeTotalEntropy \\
    \ExampleThreeSpeciesOnePDF
    \ExampleThreeSpeciesTwoPDF
    \caption{Example \ref{Example:BKW mass ratio 100}:  The upper left and right plots are the time evolution of the total energy and the total entropy.  The lower left and right are plots of cross-sections of the particle solution compared to the exact solution, with species 1 on the left and species 2 on the right at time $t = 5$.}
    \label{fig:Example 3 total energy and entropy and PDFs}
\end{figure}
\subsection{Coulomb Example 1}\label{Example:Coulomb case example 2} 
The final two examples are Coulomb collision examples.  In the first case, the masses of each species are the same $m_1 = m_2 = 1$, and $B_{11} = B_{12} = B_{21} = B_{22} = \frac{1}{32}$. The number of particles used is $n^2 = 50^2$.  
Figure \ref{fig:Example 5 Momentum total energy and total entropy} shows the time evolution of total energy and total entropy for different values of $\Delta t$.  The total energy is conserved on the order of $\Delta t$ and that the total entropy is decreasing.  Figure \ref{fig:Example 5 velocity and temperature} shows the time evolution of the species velocities (in each dimension) and species temperatures using $\Delta t = 0.02$.  The velocities and temperatures of each species indeed relax to the expected equilibrium velocities and temperature.  The size of the computational domain is the same for both species, that is $L_1 = L_2 = 4$.  This along with $m_1 = m_2$ ensures that \eqref{eq:computational domain constraint} is satisfied.

\newcommand{\ExampleFiveEnergy}{\includegraphics[width=.49\textwidth]{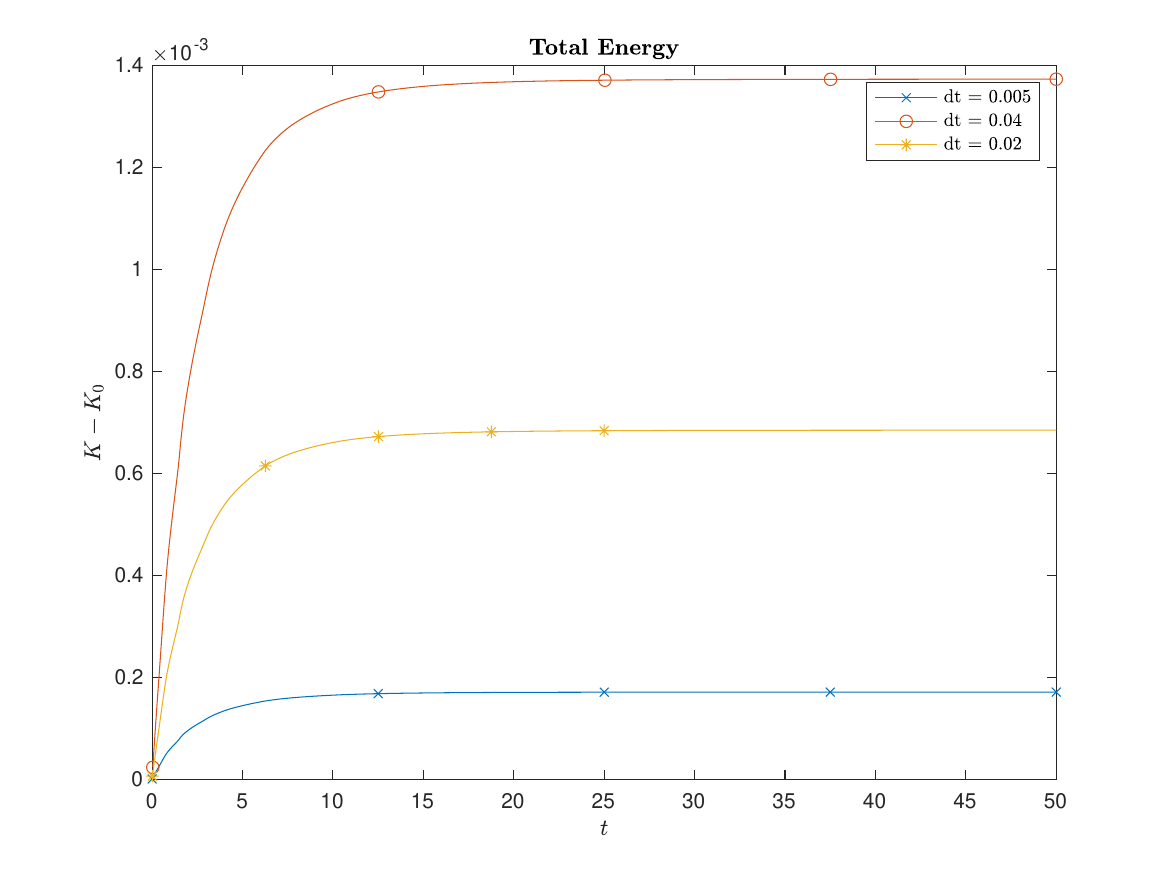}}

\newcommand{\ExampleFiveEntropy}{\includegraphics[width=.49\textwidth]{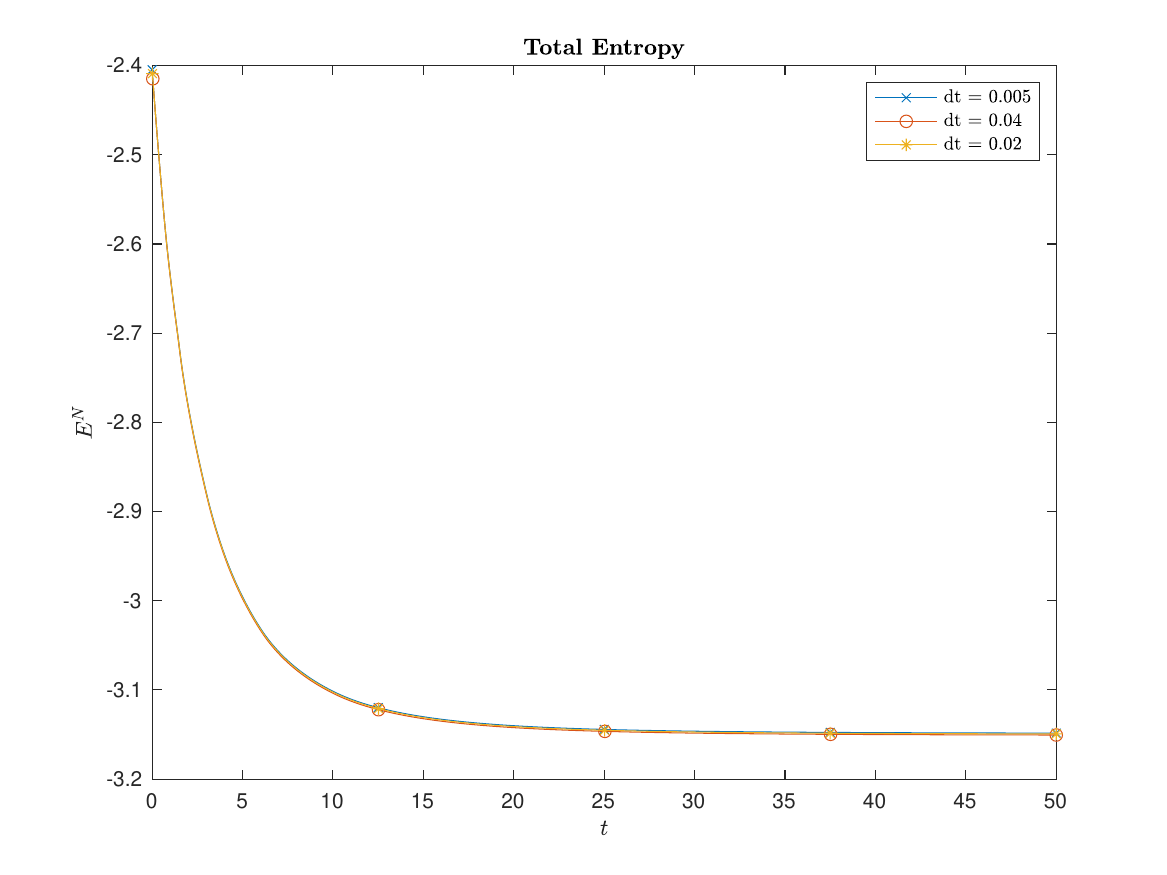}}

\newcommand{\ExampleFiveTemperature}{\includegraphics[width=.49\textwidth]{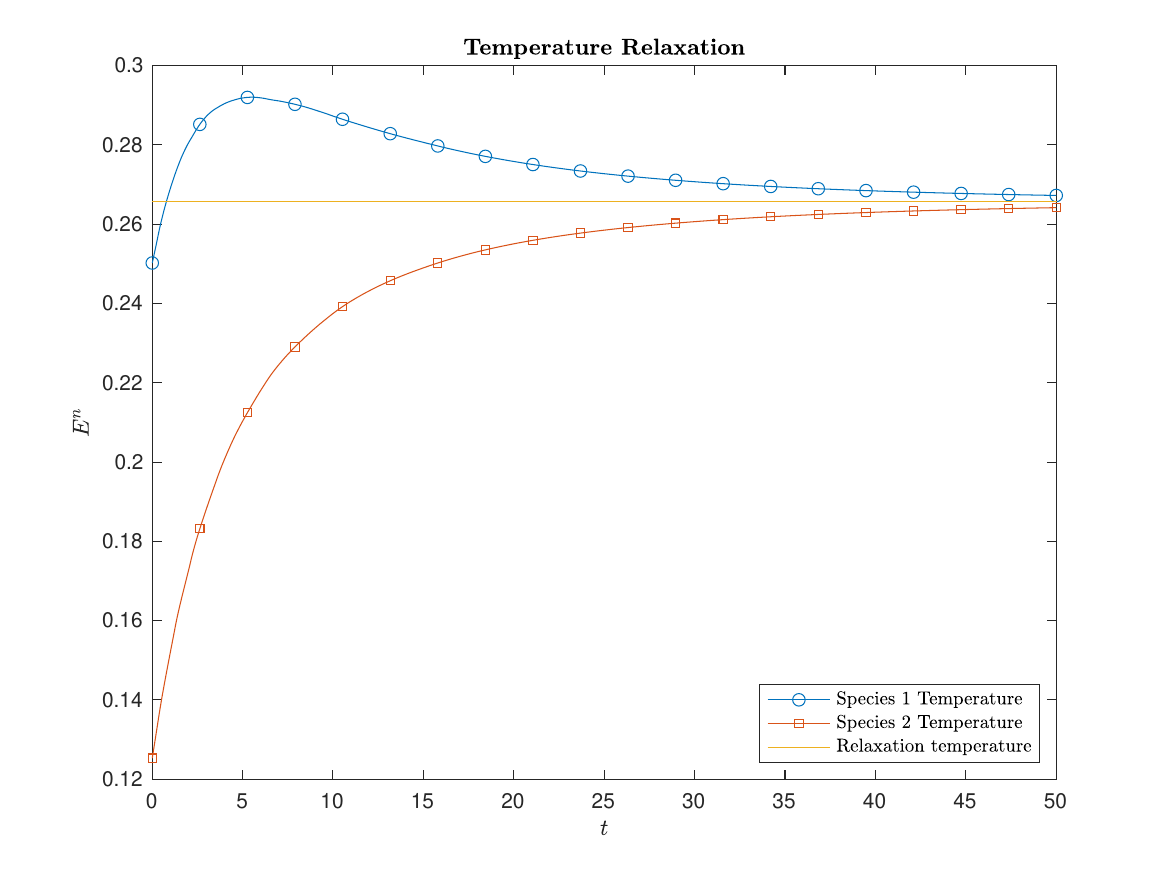}}

\newcommand{\ExampleFiveVelocityVOne}
{\includegraphics[width=.49\textwidth]{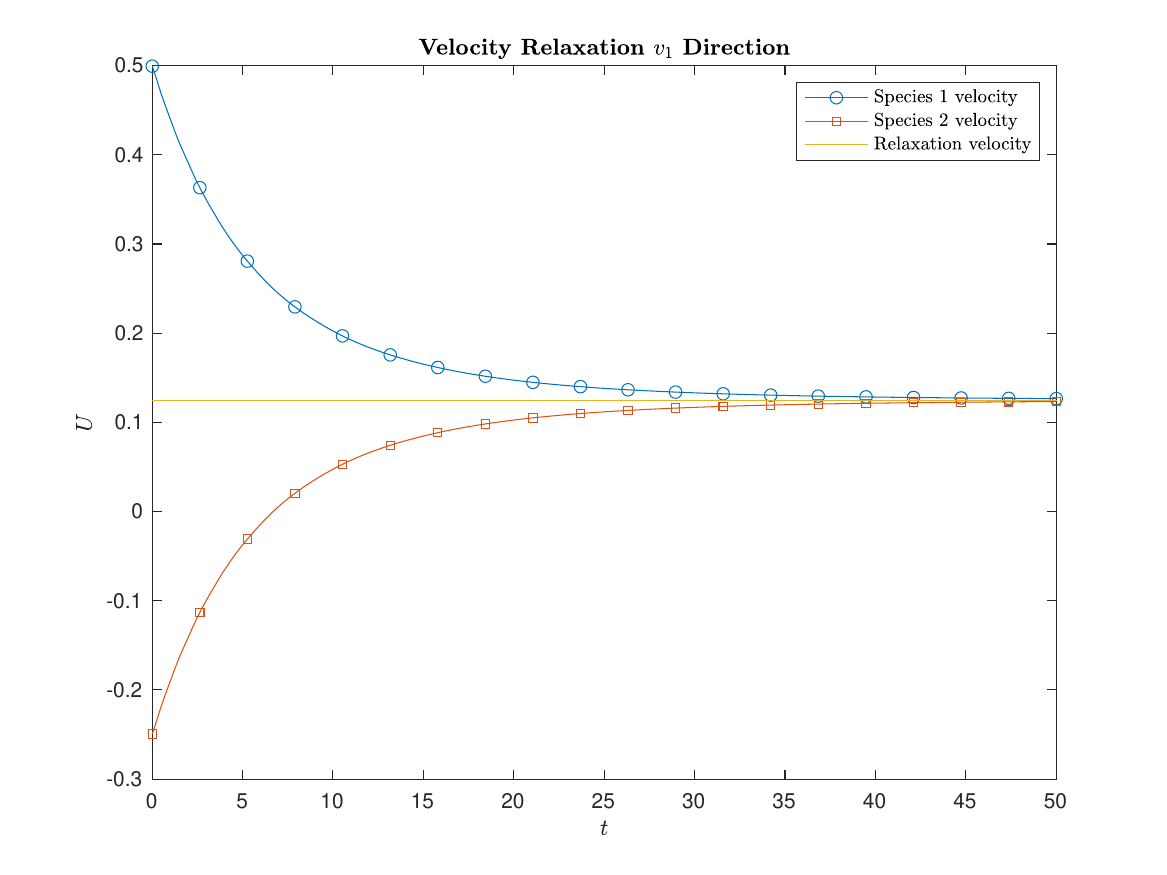}}

\newcommand{\ExampleFiveVelocityVTwo}
{\includegraphics[width=.49\textwidth]{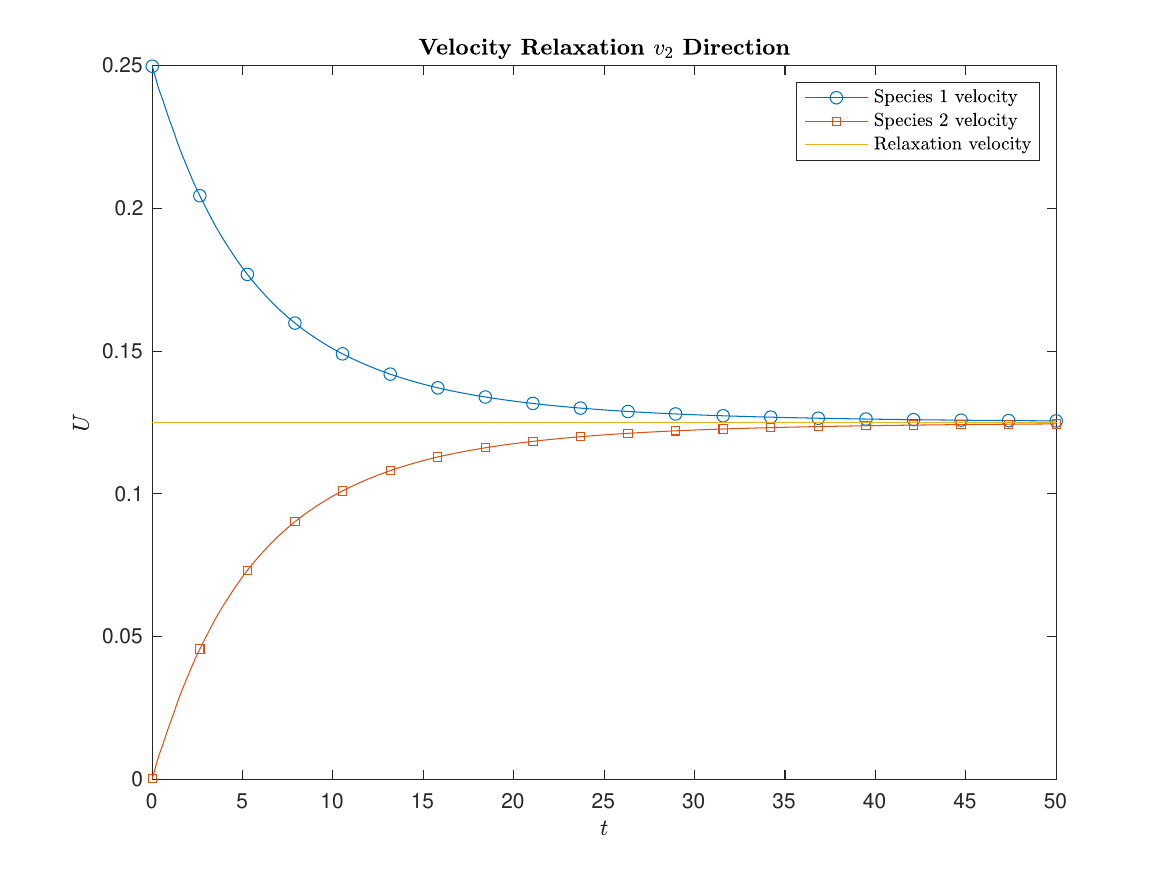}}

\newcommand{\ExampleFiveMomentumVOne}{\includegraphics[width=.49\textwidth]{Example_5_Total_Momentum_v1_direction_dt_test.eps}}

\newcommand{\ExampleFiveMomentumVTwo}{\includegraphics[width=.49\textwidth]{Example_5_Total_Momentum_v2_direction_dt_test.eps}}

\begin{figure}[htp]
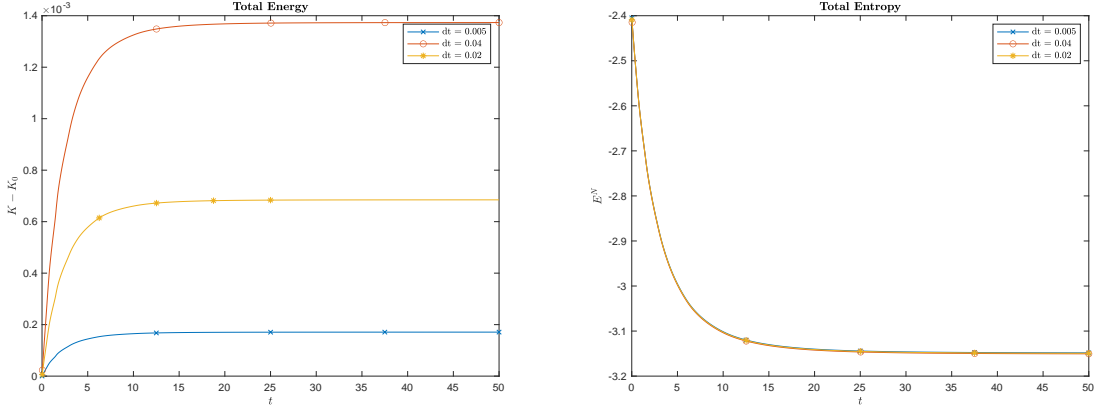

    \centering  
    \ExampleFiveEnergy
    \ExampleFiveEntropy
    \caption{Example \ref{Example:Coulomb case example 2}: The left and right plots are the time evolution of the total energy and the total entropy.}
    \label{fig:Example 5 Momentum total energy and total entropy}
\end{figure}

\begin{figure}[htp]
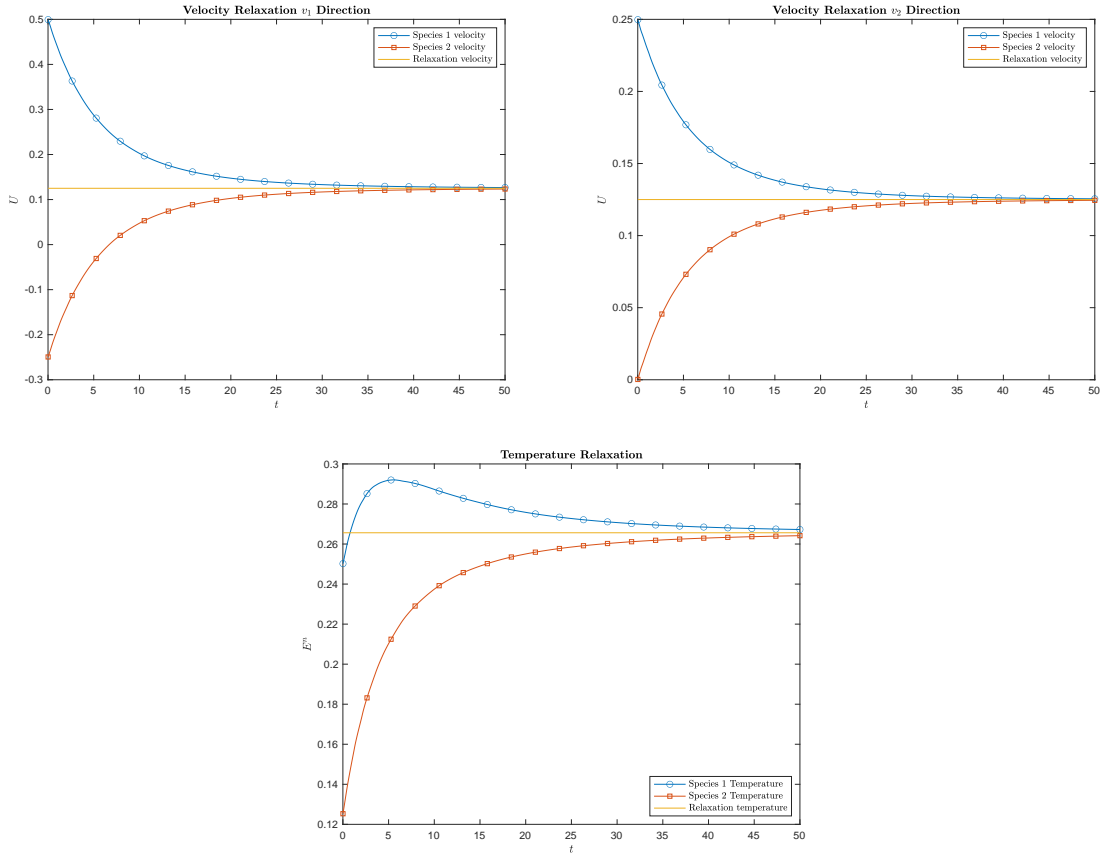

    \centering
    \ExampleFiveVelocityVOne
    \ExampleFiveVelocityVTwo \\
    \ExampleFiveTemperature
    \caption{Example \ref{Example:Coulomb case example 2}:  The left and right in the upper row are plots of the time evolution of the velocities of both species in each dimension, along with the equilibrium velocity.  The bottom is a plot of the time evolution of the temperature of both species, along with the equilibrium temperature.}
    \label{fig:Example 5 velocity and temperature}
\end{figure}

\subsection{Coulomb Example 2}\label{Example:Coulomb case example 3} 
In this example the masses are $m_1 = 2$ and $m_2 = 1$ and $B_{11} = \frac{1}{8}$, $B_{12} = B_{21} = \frac{1}{16}$, and $B_{22} = \frac{1}{16}$.
The number of particles used is $n^2 = 60^2$.  We compare the results from using the same computational domain size $L_1 = L_2 = 4$ for both species to the results from using a different computational domain size for each species $L_1 = 2.5$ and choosing $L_2$ so that the constraint \eqref{eq:computational domain constraint} is satisfied.  Figure \ref{fig:Example 6 Temperature and velocity} shows the time evolution of the temperature and velocity for each species.  Using different domain sizes for each species, the temperatures for each species relax to a species independent equilibrium temperature, while using the same computational domain for both species, the temperatures relax to species dependent equilibrium temperatures which is unphysical. Figure \ref{fig:Example 6 Energy and Entropy} shows the time evolution of the total energy and entropy using the same and different computational domain sizes. From these figures, it is clear that the constraint \eqref{eq:computational domain constraint} is critical to guarantee the correct relaxation of temperature (while this failure may not be visible in other quantities).

\newcommand{\ExampleSixEnergy}{\includegraphics[width=.49\textwidth]{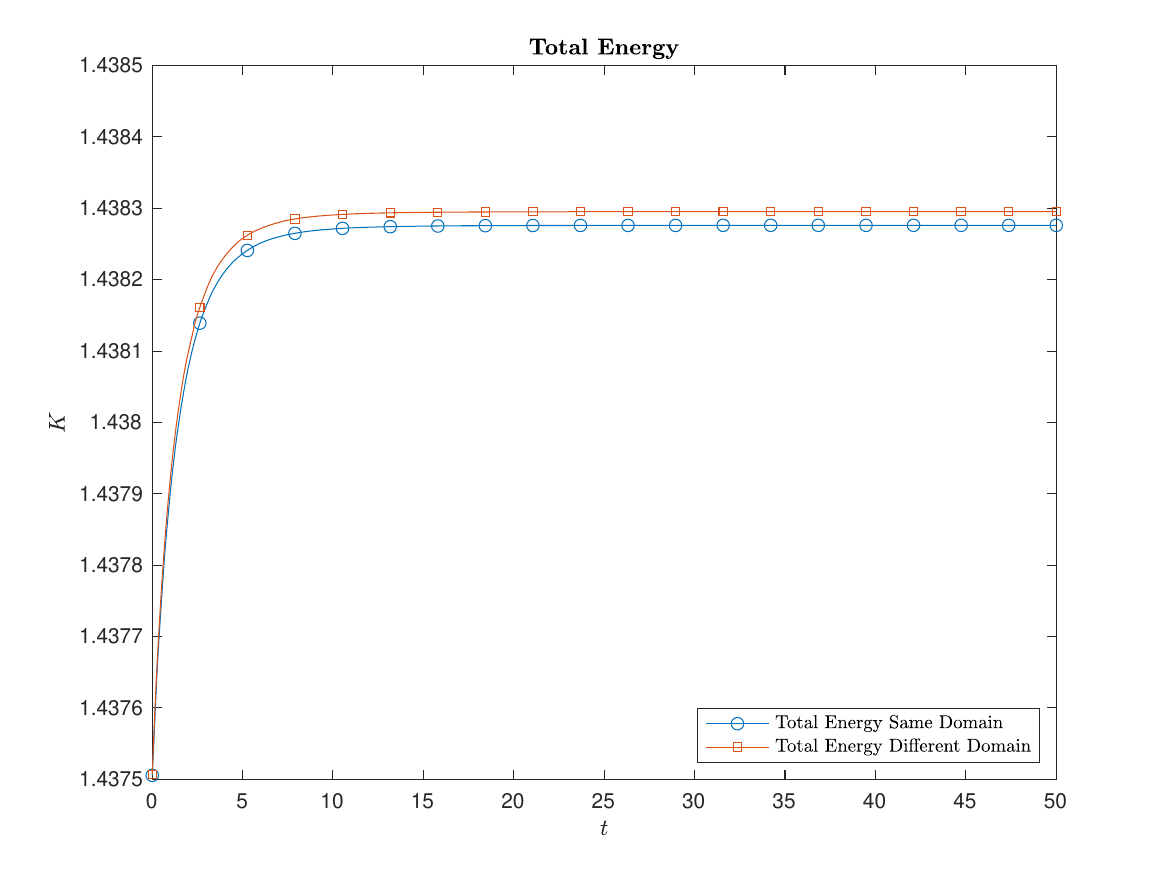}}

\newcommand{\ExampleSixEntropy}{\includegraphics[width=.49\textwidth]{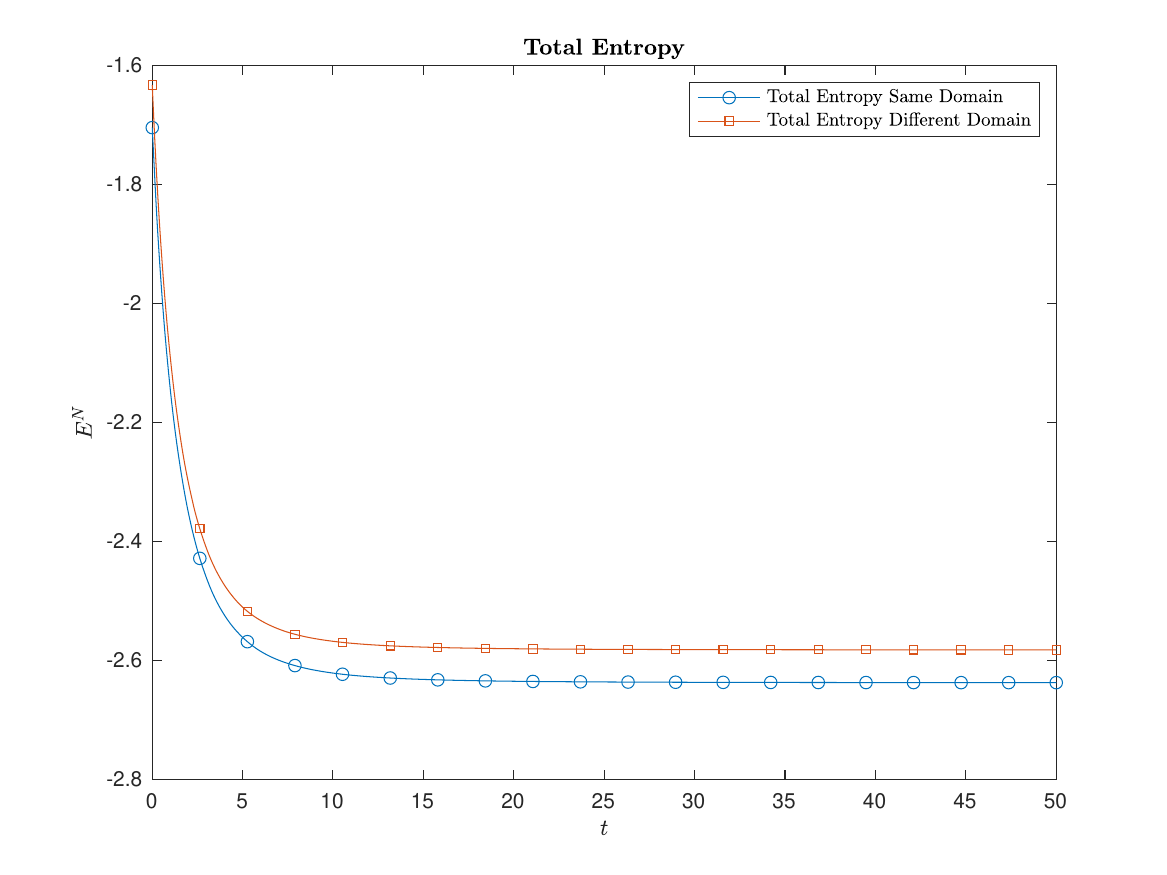}}

\newcommand{\ExampleSixVelocityVOne}{\includegraphics[width = .49\textwidth]{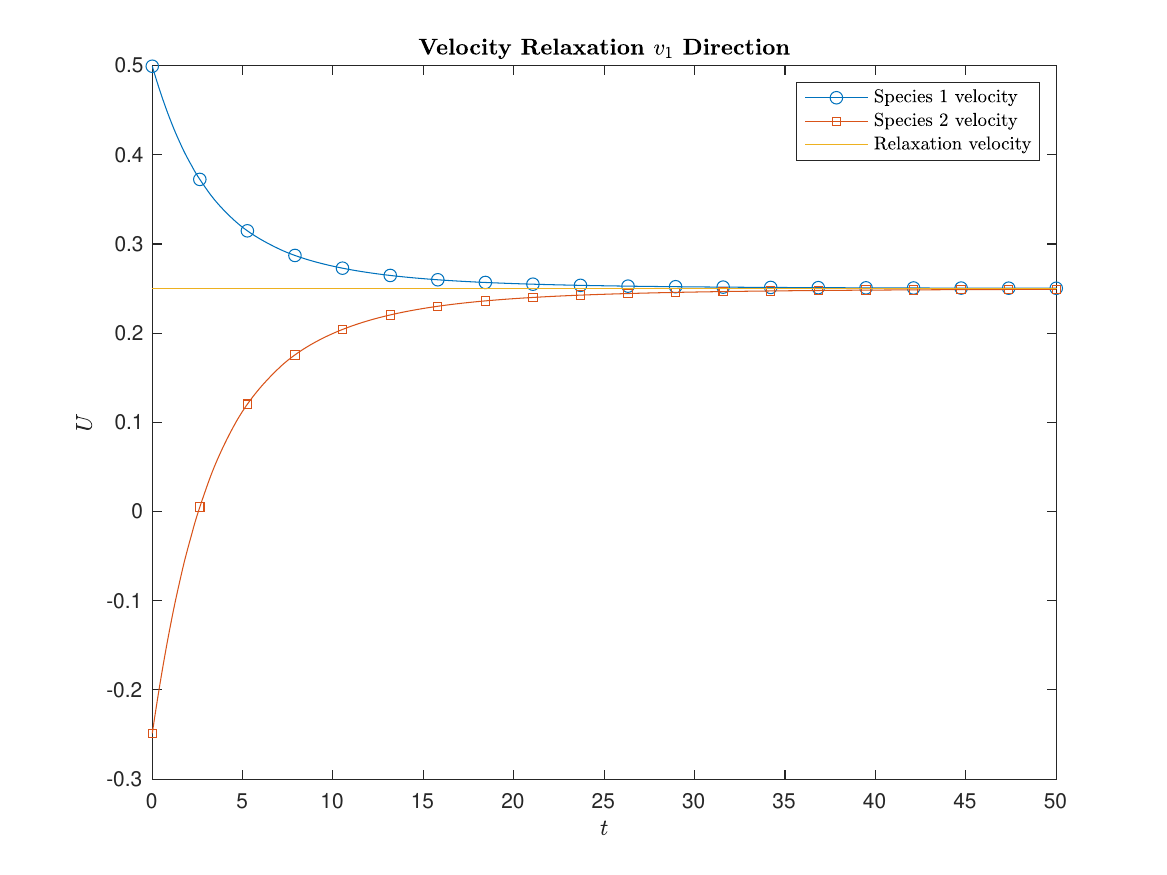}}

\newcommand{\ExampleSixVelocityTwo}{\includegraphics[width = .49\textwidth]{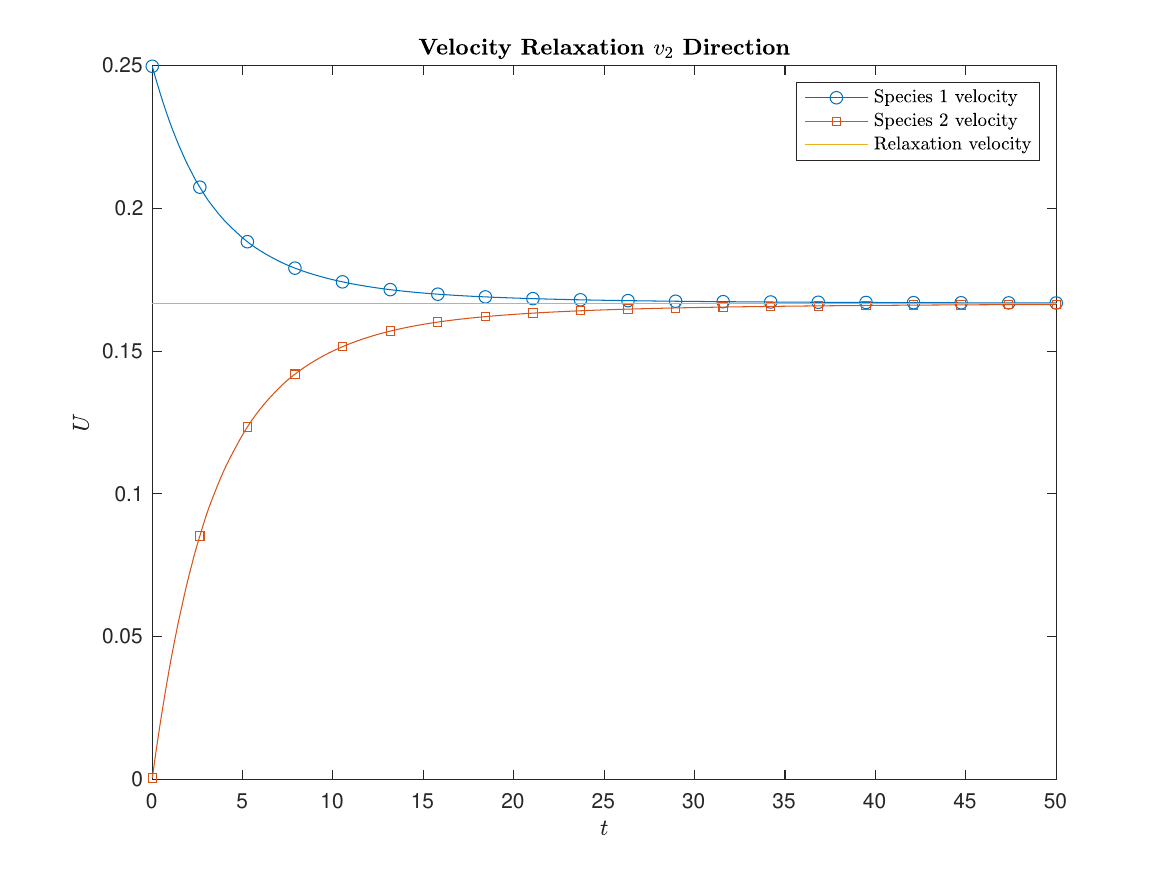}}

\newcommand{\ExampleSixTemperatureSame}{\includegraphics[width = .49\textwidth]{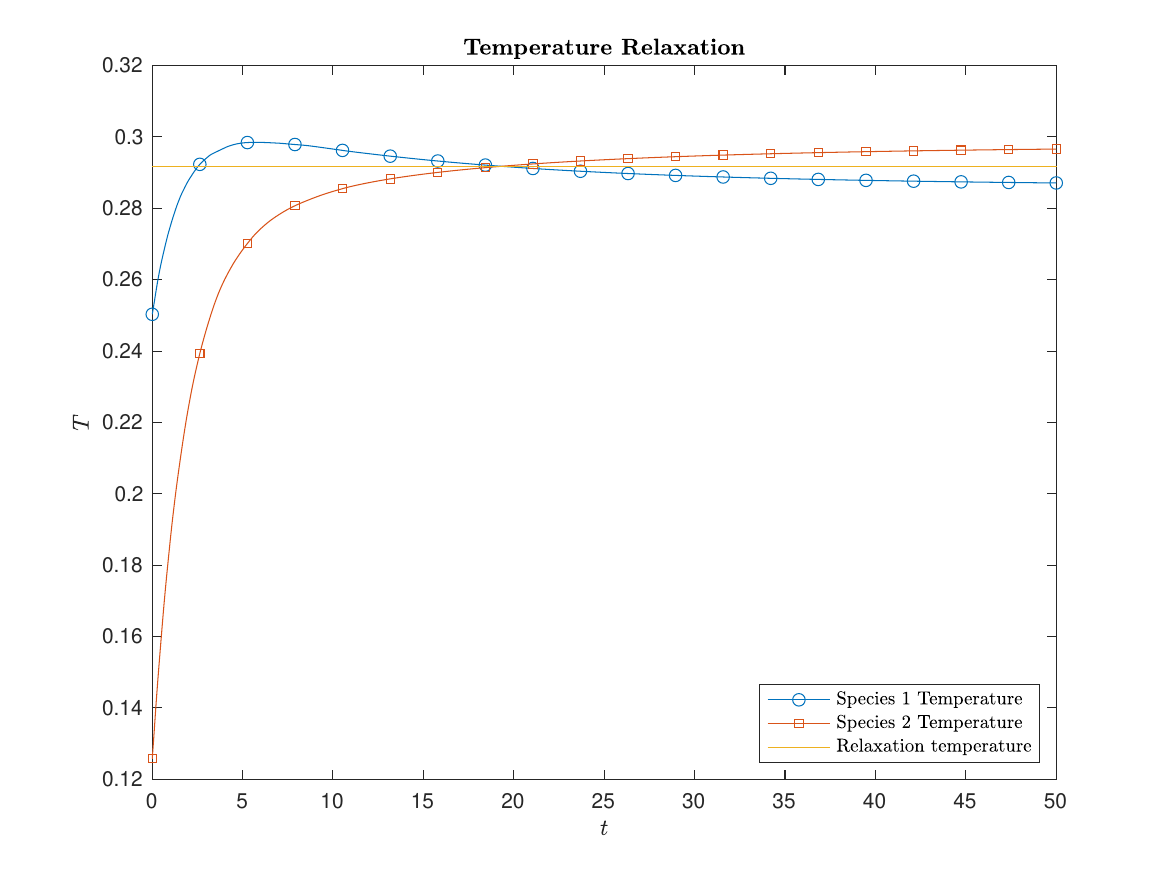}}

\newcommand{\ExampleSixTemperatureDifferent}{\includegraphics[width = 0.49\textwidth]{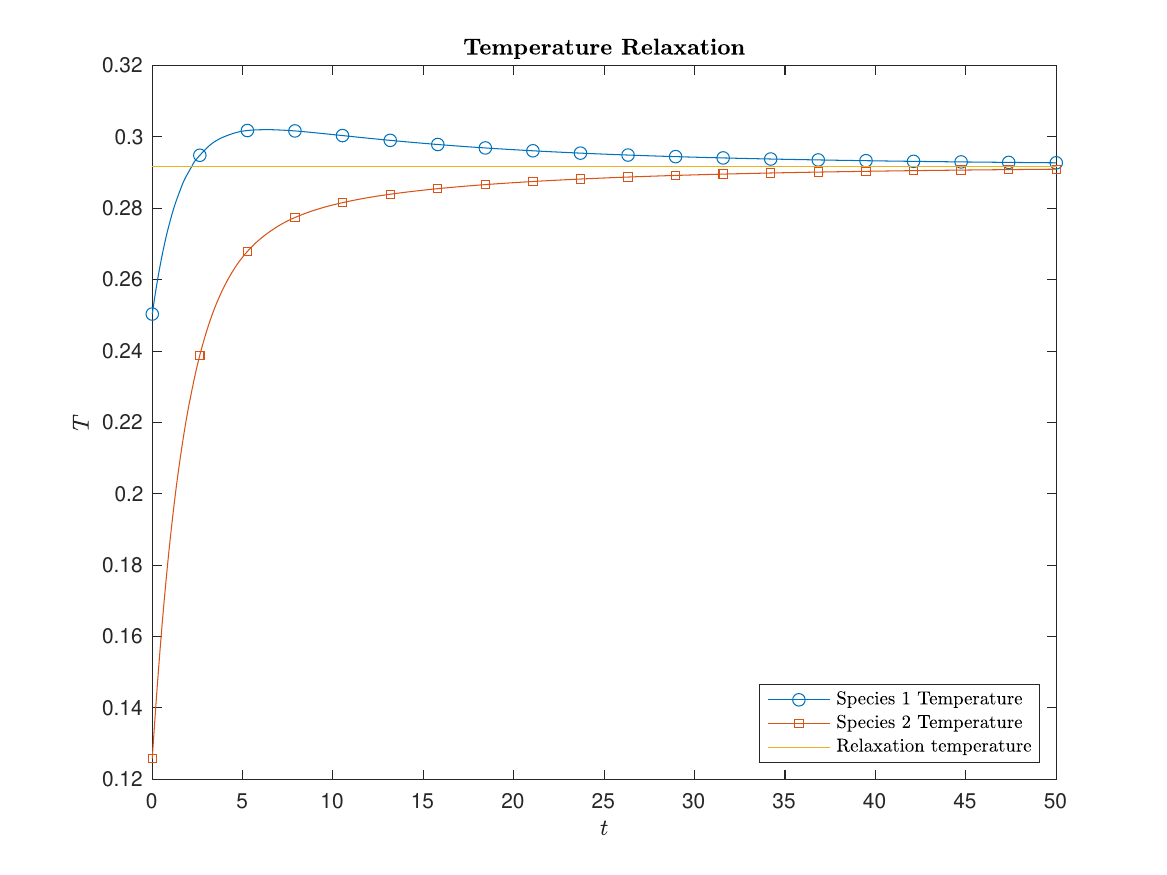}}

\begin{figure}[htp]
    \centering
    \ExampleSixTemperatureSame
    \ExampleSixTemperatureDifferent \\
    \ExampleSixVelocityVOne
    \ExampleSixVelocityTwo
    \caption{Example \ref{Example:Coulomb case example 3}:  The left and right plots in the upper row are  the time evolution of temperature of both species, along with the equilibrium temperature.  On the left the same computational domain is used for each species and on the right a different computational domain is used for each species.  The left and right plots in the bottom row are of the velocity relaxation, which is the same when using different computational domains versus using the same computational domain.}
    \label{fig:Example 6 Temperature and velocity}
\end{figure}

\begin{figure}[htp]
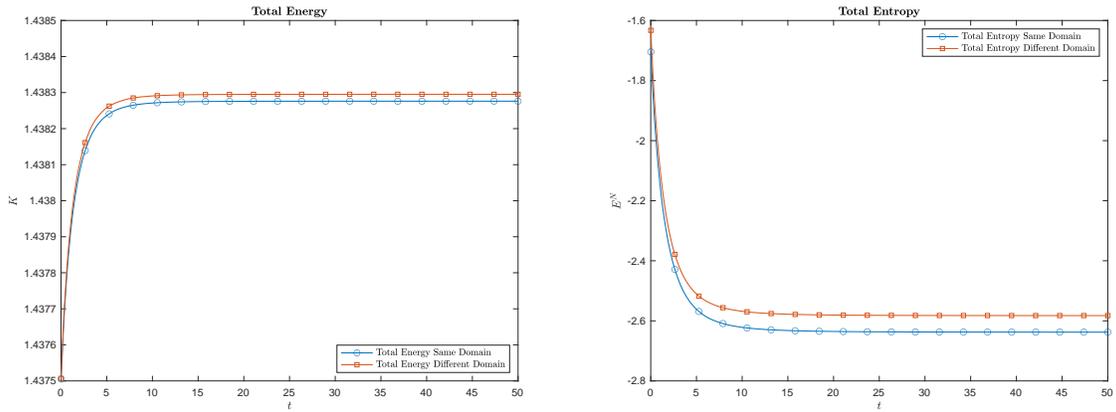

    \centering
    \ExampleSixEnergy
    \ExampleSixEntropy
    \caption{Example \ref{Example:Coulomb case example 3}:  The left and right plots are the time evolution of the total energy and the total entropy using the same computational domain for each species and a different computational domain for each species.}
    \label{fig:Example 6 Energy and Entropy}
\end{figure}

\section{Conclusions and future work}\label{sec:conclusion}
    In this work, we presented a deterministic particle method for the spatially homogeneous multispecies Landau equation. We first regularized the Landau equation to give access to the particle approximation. We showed that this regularized equation conserves total mass, momentum, and energy along with a decaying regularized entropy. It also has a Maxwellian equilibrium distribution and, a direct result of this, is that requiring $m_i \epsilon_i = \epsilon$ for all species forces the equilibrium temperature to be species independent.  At the semi-discrete level (continuous in time), we showed that the particle method inherits the conservation and entropy decay properties in the continuous case.  At the fully discrete level, we examined the use of the forward Euler and implicit midpoint method. Extensive numerical examples were presented to illustrate the accuracy and structure-preserving properties of the proposed particle method.
    



    There are several possible directions for future work regarding this project. First of all, it is natural to couple this method with the particle-in-cell (PIC) method to handle the full Vlasov-Landau equation \eqref{eq:multispecies Vlasov-Landau equation}, which is currently the predominate method used to solve collisionless plasma problems. Further, the particle method, as in its current form, is quite expensive especially in higher dimensions. To reduce the computational cost, a promising way is to apply the random batch method \cite{CJT22} while still retaining the major physical properties. Finally it would be interesting to explore an efficient iterative scheme for the implicit time stepping method.

\label{sec:con}



\section*{Appendix: A BKW solution to the multispecies Landau equation}
\label{app:BKW solution}
An exact solution to the multispecies Boltzmann equation was provided in \cite{KW1977} and here we use a similar strategy to find an exact solution to the multispecies Landau equation \eqref{eq:homogenous Landau equation} in the Maxwell collision case.  We further assume that the kernel of the Landau collision operator can be written as
\begin{equation*}
A_{ji}(z)=\frac{|\ln \delta|}{8\pi \varepsilon_0^2}\frac{q_i^2q_j^2}{m_i}(|z|^2I_d-z\otimes z)=\frac{B_{ij}}{m_i}(|z|^2 I_d-z\otimes z), \quad B_{ij} = \frac{|\ln \delta|}{8\pi \varepsilon_0^2}q_i^2q_j^2.
\end{equation*}
We start by looking for solutions with the ansatz 
\begin{equation*}
    f_i(t,v) = n_i\left(\frac{m_i}{2\pi K} \right)^{\frac{d}{2}}\exp{\left(-\frac{m_iv^2}{2K} \right)}\left(P_i + \frac{m_i}{K}Q_iv^2\right),
\end{equation*}
where $K = K(t)$ is to be found.  We begin with a calculation of the particle number densities and a normalization of bulk temperature
\begin{equation*}
    n_i = \int_{\mathbb{R}^d}f_i\rd{v} = n_i\left(P_i + dQ_i\right), \quad
    T = \frac{1}{dn}\sum_{i=1}^sm_i\int_{\mathbb{R}^d}f_iv^2\rd{v} = \frac{K}{n}\sum_{i=1}^sn_i(1+2Q_i) = 1.
\end{equation*}
From the temperature equation,
\begin{equation*}
    K = \frac{n}{n + 2\sum_{i=1}^s n_iQ_i},
\end{equation*}
and from the particle number densities equation $P_i = 1-dQ_i$,  which allows for the solution ansatz to be rewritten as 
\begin{equation}\label{eqn:BKW ansatz}
    f_i(t,v) = n_i\left(\frac{m_i}{2\pi K} \right)^{\frac{d}{2}}\exp{\left(-\frac{m_iv^2}{2K} \right)}\left(1-dQ_i + \frac{m_i}{K}Q_iv^2\right).
\end{equation}
Using this ansatz, we calculate the right-hand side of the Landau equation beginning with 
\begin{equation*}
    \nabla_{v}f_i = n_i\left(\frac{m_i}{2\pi K}\right)^{\frac{d}{2}}\exp{\left(-\frac{m_iv^2}{2K} \right)}\frac{m_i}{K}\left(2Q_i - \left(1-dQ_i+\frac{m_i}{K}Q_iv^2 \right)\right)v,
\end{equation*}
and thus,
\begin{equation*}
    \frac{1}{m_i}\nabla_v\log{f_i} = \frac{1}{m_i}\frac{\nabla_vf_i}{f_i} = \frac{1}{K}\frac{2Q_i}{1-dQ_i+\frac{m_i}{K}Q_iv^2}v - \frac{1}{K}v.
\end{equation*}
Therefore, we conclude that
\begin{align*}
    &\frac{1}{m_i}\nabla_v\log{f_i(v)} - \frac{1}{m_j}\nabla_{v_{*}}\log{f_j(v_{*})} \\
    &= \frac{2}{K}\frac{(Q_i - Q_j)v + Q_j(v-v_*) - dQ_iQ_j(v-v_*)}{\left((1-dQ_i)+\frac{m_i}{K}Q_iv^2 \right)\left((1-dQ_j)+\frac{m_j}{K}Q_jv_{*}^2 \right)}\\
    & \quad +\frac{2}{K}\frac{\frac{1}{K}Q_iQ_j\left((m_jv_*^2-m_iv^2)v+m_iv^2(v-v_*) \right)}{\left((1-dQ_i)+\frac{m_i}{K}Q_iv^2 \right)\left((1-dQ_j)+\frac{m_j}{K}Q_jv_{*}^2 \right)} - \frac{1}{K}(v-v_{*}).
\end{align*}
Because $A_{ji}(z)z = 0$, 
\begin{align*}
    A_{ji}&(v-v_{*})\left[\frac{1}{m_i}\nabla_v\log{f_i(v)} - \frac{1}{m_j}\nabla_{v_{*}}\log{f_j(v_{*})}\right] \\
    &= \frac{2}{K}\frac{(Q_i - Q_j)A_{ji}(v-v_{*})v + \frac{1}{K}Q_iQ_j\left(m_jv_{*}^2-m_iv^2\right)A_{ji}(v-v_{*})v}{\left((1-dQ_i)+\frac{m_i}{K}Q_iv^2 \right)\left((1-dQ_j)+\frac{m_j}{K}Q_jv_{*}^2 \right)},
\end{align*}
and
\begin{align*}
    &A_{ji}(v-v_{*})\left[\frac{1}{m_i}\nabla_v\log{f_i(v)} - \frac{1}{m_j}\nabla_{v_{*}}\log{f_j(v_{*})}\right]f_i(v)f_j(v_{*}) \\
    &= \frac{2}{K}n_in_j\left(\frac{m_i}{2\pi K}\right)^{\frac{d}{2}}\left(\frac{m_j}{2\pi K}\right)^{\frac{d}{2}}\exp{\left(-\frac{m_iv^2}{2K}\right)} \exp{\left(- \frac{m_jv_{*}^2}{2K} \right)} \\
    & \quad\times\left((Q_i - Q_j)A_{ji}(v-v_{*})v + \frac{1}{K}Q_iQ_j\left(m_jv_{*}^2-m_iv^2\right)A_{ji}(v-v_{*})v\right).
\end{align*}
To continue we need to integrate both sides of the equation above w.r.t $v_*$
\begin{align*}
    &\int_{\mathbb{R}^d} A_{ji}(v-v_{*})\left[\frac{1}{m_i}\nabla_v\log{f_i(v)} - \frac{1}{m_j}\nabla_{v_{*}}\log{f_j(v_{*})}\right]f_i(v)f_j(v_{*}) \rd{v_*}\\
    &= \frac{2}{K}n_in_j\left(\frac{m_i}{2\pi K} \right)^{\frac{d}{2}}\exp{\left(-\frac{m_iv^2}{2K} \right)}(I_1 + I_2),
\end{align*}
where 
\begin{align*}
    I_1 &= (Q_i-Q_j)\left(\frac{m_j}{2\pi K}\right)^{\frac{d}{2}}\int_{\mathbb{R}^d}\exp{\left(-\frac{m_jv_{*}^2}{2K} \right)}A_{ji}(v-v_*)v\rd{v_*},\\
    I_2 &= \frac{1}{K}Q_iQ_j\left(\frac{m_j}{2\pi K}\right)^{\frac{d}{2}}\int_{\mathbb{R}^d}\exp{\left(-\frac{m_jv_{*}^2}{2K} \right)}A_{ji}(v-v_*)\left(m_jv_{*}^2-m_iv^2\right)v\rd{v_*}.
\end{align*}
After some calculations we see that 
\begin{align*}
    I_1 &= K(Q_i-Q_j)\frac{B_{ij}}{m_im_j}(d-1)v,\\
    I_2 &= Q_iQ_j\frac{B_{ij}}{m_im_j}(d-1)\left((d+2)K-m_iv^2\right)v,
\end{align*}
and therefore, we conclude that 
\begin{align*}
&\int_{\mathbb{R}^d} A_{ji}(v-v_{*})\left[\frac{1}{m_i}\nabla_v\log{f_i(v)} - \frac{1}{m_j}\nabla_{v_{*}}\log{f_j(v_{*})}\right]f_i(v)f_j(v_{*}) \rd{v_*}\\
    &= \frac{2}{K}n_in_j\left(\frac{m_i}{2\pi K} \right)^{\frac{d}{2}}\exp{\left(-\frac{m_iv^2}{2K} \right)}(d-1)\frac{B_{ij}}{m_im_j}\left(K(Q_i-Q_j)+Q_iQ_j\left((d+2)K-m_iv^2\right)\right)v.
\end{align*}
Finally, the right-hand side of the Landau equation (\ref{eq:homogenous Landau equation}) reads
\begin{equation}\label{eq:RHS}
\begin{aligned} 
&\sum_{j=1}^{s}\nabla_{v} \cdot \int_{\mathbb{R}^d} A_{ji}(v-v_{*})\left[\frac{1}{m_i}\nabla_v\log{f_i(v)} - \frac{1}{m_j}\nabla_{v_{*}}\log{f_j(v_{*})}\right]f_i(v)f_j(v_{*}) \rd{v_*}\\
    &
    \begin{aligned}
    =\sum_{j=1}^s&\frac{2}{K}n_in_j\left(\frac{m_i}{2\pi K}\right)^{\frac{d}{2}}(d-1)\frac{B_{ij}}{m_im_j}\exp{\left(-\frac{m_iv^2}{2K} \right)}\\
    &\times \left(Q_iQ_j\left(v^4\frac{m_i^2}{K}-2m_i(d+2)v^2+d(d+2)K\right)  + (Q_i-Q_j)\left( dK - m_iv^2\right) \right).
    \end{aligned}
\end{aligned}
\end{equation}
For the calculation of the left-hand side of (\ref{eq:homogenous Landau equation}), taking the derivative w.r.t time of \eqref{eqn:BKW ansatz} yields
\begin{equation}\label{eq:LHS}
\begin{aligned}
    \partial_tf_i = \;&n_i\left( \frac{m_i}{2\pi K}\right)^{\frac{d}{2}}\exp{\left(-\frac{m_iv^2}{2K} \right)}\\
    &  \times
    \left(\left(\frac{m_iv^2}{2K^2}K^{\prime} - \frac{d}{2K}K^{\prime}\right)\left(1-dQ_i + \frac{m_i}{K}Q_iv^2\right)-dQ_i^{\prime}+\frac{m_i}{K}Q_i^{\prime}v^2-\frac{m_i}{K^2}Q_iK^{\prime}v^2\right).
\end{aligned}
\end{equation}

Here for simplicity we consider two-species systems $(s=2)$ and assume $Q_1 = Q_2 = Q$.  With this simplification we can solve for $Q$ in terms of $K$ 
\begin{equation*}
    K = \frac{n_1 + n_2}{n_1 + n_2 + 2(n_1+n_2)Q}=\frac{1}{1+2Q} \quad \Longleftrightarrow \quad Q = \frac{1-K}{2K}
    \quad \mbox{and} \quad Q^{\prime} = \frac{-K^{\prime}}{2K^2}.
\end{equation*}
Then \eqref{eq:RHS} simplifies to 
\begin{equation*}
    \frac{n_i}{m_i}\left(\frac{m_i}{2\pi K} \right)^{\frac{d}{2}}(d-1)\exp{\left(-\frac{m_iv^2}{2K}\right)}\frac{(1-K)^2}{2K^4} \left(m_i^2v^4 - 2m_i(d+2)Kv^2 + d(d+2)K^2 \right)\sum_{j=1}^2\frac{B_{ij}}{m_j}n_j,
\end{equation*}
while (\ref{eq:LHS}) simplifies to 
\begin{equation*}
n_i\frac{(1-K)}{4K^4}\left(\frac{m_i}{2 \pi K} \right)^{\frac{d}{2}}\exp{\left(-\frac{m_iv^2}{2 K} \right)}\left(m_i^2v^4 - 2(d+2)m_iKv^2 + d(d+2)K^2 \right)K^{\prime}.
\end{equation*}
Matching the above two equations, we have 
\begin{equation*}
    K^{\prime} = \frac{2}{m_i}(d-1)(1-K)\sum_{j=1}^2\frac{B_{ij}}{m_j}n_j.
\end{equation*}
We further require that 
$$
\sum_{j=1}^2 \frac{B_{ij}}{m_im_j}n_j := \beta_i, \quad \mbox{and} \quad \beta_1 = \beta_2 = \beta,
$$ and $\beta$ is some constant. Then,
$$K'=2\beta(d-1)(1-K),$$
which results in $K=1-C\exp(-2\beta(d-1)t)$, where $C$ is a constant of integration.  For the 2D BKW examples, we choose $C=1/2$ and $\beta = 1/16$, then $K = 1-\exp(-t/8)/2$.



\bibliographystyle{plain}
\bibliography{hu_bibtex}

\end{document}